\documentclass[12pt,epsfig,amsfonts]{amsart} 
\usepackage{amsmath,amsthm,amssymb,amscd,epsfig,color}
\usepackage{graphicx}
\usepackage{mathrsfs}
\usepackage{overpic}
\usepackage{ulem}
\newcommand\cA{\bar {A}}

\newtheorem*{theorema}{Theorem A}

\newtheorem*{eq:abnormallowerbound}{Lemma B.2}
\newtheorem*{theoremb}{Theorem B}

\newtheorem{prop}{Proposition}[section]
\newtheorem{lemma}[prop]{Lemma}

\newtheorem{sublemma}[prop]{Sublemma}

\newtheorem{cor}[prop]{Corollary}

\theoremstyle{remark}

\newtheorem{remark}[prop]{Remark}

\numberwithin{equation}{section}

\begin{document}

\author{Hiroki Takahasi and Masato Tsujii}

\address{Keio Institute of Pure and Applied Sciences (KiPAS), Department of Mathematics,
Keio University, Yokohama,
223-8522, JAPAN} 
\email{hiroki@math.keio.ac.jp}

\address{Department of Mathematics, Kyushu University, Fukuoka, 819-0395, JAPAN} 
\email{tsujii@math.kyushu-u.ac.jp}


\title[Existence of Large deviation rate function]{
Existence of Large deviations rate function\\ for any $S$-unimodal map}

 \dedicatory{Dedicated to Professor Emeritus Hiroshi Kokubu\\ on the occasion of his retirement by age 65}

   \begin{abstract}
For an arbitrary negative Schwarzian unimodal map with a non-flat critical point,
    we establish the level-2 Large Deviation Principle for empirical distributions.
We also give an example of a bimodal map for which the level-2 Large Deviation Principle does not hold. 
   \end{abstract}
   \maketitle

   \section{Introduction}

A main objective of the ergodic theory of smooth dynamical systems is to understand long-term behavior of typical orbits for a majority of systems. 
   Much effort has been dedicated to 
   constructing physically relevant invariant measures which statistically predict
   typical asymptotic behaviors. 
      A refined description of the dynamics requires the analysis of
    atypical, or transient behaviors before orbits settle down 
   to equilibrium. The theory of large deviations 
    is concerned with such rare events.
             The Large Deviation Principle (LDP) asserts the existence of the rate function
      which controls probabilities of rare events on exponential scale.

      It is now classical that a transitive uniformly hyperbolic (Axiom A) attractor supports a unique
            Sina{\u\i}-Ruelle-Bowen measure \cite{Bow75,Sin72}, and Lebesgue almost every orbit in the basin of attraction is asymptotically 
            distributed with respect to this measure.
            The LDP for Axiom A attractors was established by Kifer \cite{Kif90}, Orey $\&$ Pelikan \cite{OrePel89} and Takahashi \cite{Tak84}. 
            For one-dimensional non-hyperbolic systems, 
        after several progresses \cite{Chu11,PolSha09,PolShaYur}, a major advance was made in
            \cite{CRT} which establishes the LDP
             for an arbitrary $C^{1+\alpha}$ multimodal map with non-flat critical points
 that is topologically exact.
The aim of this paper is to treat
what is left off in \cite{CRT}:
 the LDP for renormalizable unimodal maps, including infinitely renormalizable ones.
 The conclusion is that the LDP holds for an arbitrary $S$-unimodal map
 with a non-flat critical point.

We introduce our setting and results in more precise terms.
Let $X\subset\mathbb R$ be a compact non-degenerate interval. 
   A $C^1$ map $f\colon X\to X$ is called {\it unimodal} 
          if it has a unique critical point $c$, which is contained in ${\rm int}(X)$ and is an extremum, and satisfies $f(\partial X)\subset\partial X$.
  An {\it $S$-unimodal map} $f$ is a unimodal map of class $C^3$ on $X\setminus\{c\}$ with 
    negative Schwarzian derivative $f'''/f'-(3/2)(f''/f')^2<0$
  such that if $x\in\partial X$ is a fixed point of $f$ then $|f'(x)|>1$.
We say the critical point $c$ is {\it non-flat} if there exist $l>1$ and 
$C^3$ diffeomorphisms $\varphi$ and $\psi$ defined on a neighborhood of $c$ and $f(c)$ 
respectively such that $\varphi(c)=0=\psi(f(c))$ and $|\varphi(x)|^l=|\psi(f(x))|$ for any $x$ near $c$.

Let $\mathcal M$ denote the space of Borel probability measures on $X$ endowed with the weak* topology. 
The empirical measure at time $n$ with initial point $x\in X$ is
the uniform probability distribution on the orbit $\{x,f(x),\ldots,f^{n-1}(x)\}$, denoted by
\[
\delta_x^n=\frac{1}{n}\sum_{k=0}^{n-1}\delta_{f^k(x)}\in \mathcal M,
\]
where $\delta_{f^k(x)}$ denotes the unit point mass at $f^k(x)$. 
For a Borel set $A\subset X$, we write $|A|$ for its Lebesgue measure. 
 Our main result is stated as follows.
\begin{theorema}[level-2 LDP]
For an arbitrary $S$-unimodal map $f\colon X\to X$ with a non-flat critical point,
the level-2 LDP holds, namely, there exists a convex lower 
semicontinuous function $I\colon\mathcal M\to[0,\infty]$ such that 

\smallskip
\noindent $(\text{\rm lower bound})\  \ \ \
\displaystyle{\liminf_{n\to\infty}\frac{1}{n}\log\left|\{x\in X\colon\delta_x^n\in\mathcal G\}\right|\geq-\inf_{\mathcal G} I}
$\smallskip

\noindent for any open set $\mathcal G\subset\mathcal M$, and
 \smallskip
 
\noindent$(\text{\rm upper bound})\  \ \ \
\displaystyle{\limsup_{n\to\infty}\frac{1}{n}\log\left|\{x\in X\colon\delta_x^n\in\mathcal C\}\right|\leq-\inf_{\mathcal C} I}
$\smallskip

\noindent for any closed set $\mathcal C\subset\mathcal M$.
\end{theorema}

 Hereafter we follow the convention 
   $\sup\emptyset=-\infty$, $\inf\emptyset=\infty$, $\log0=-\infty$.
The function $I$ is called {\it the rate function}.
Since $\mathcal M$ is a metrizable space, the LDP determines the rate function uniquely \cite[Theorem~2.13]{RasSep15}.

 In order to explain the meaning of Theorem~A, let us recall\footnote{Guckenheimer \cite{Guc79} proved this classification for negative Schwarzian $C^3$ unimodal maps with non-degenerate critical points. The same holds in our slightly more general setting.
 For details, see  \cite[Chapter~III, $\S$4]{dMevSt93}.}
 that an  
  $S$-unimodal map $f$ with a non-flat critical point is classified 
into the following mutually exclusive cases \cite{dMevSt93,Guc79}:
\begin{itemize}
\item[(I)] $f$ has an attracting periodic orbit.
\item[(II)] $f$ is infinitely renormalizable.
\item[(III)] $f$ is at most finitely renormalizable and has no attracting periodic orbit. 
\end{itemize}
The dynamics is relatively simple in cases (I) or (II): the empirical measure along
the orbit of Lebesgue almost every initial point converges in the weak* topology to
the measure supported on
the attracting periodic orbit or the attracting Cantor set. 
The dynamics in case (III) is much more complicated and displays
a rich array of different statistical behaviors (see e.g., \cite{BKNS96,HofKel90,Jak81,Joh87}).
As a prototypical example,
consider the quadratic map $x\in[0,1]\mapsto a x(1-x)$ with $1< a\leq 4$. The following  are well-known:
  \begin{itemize}
  \item the set of $a$-values corresponding to case (I) is open and dense in the parameter space \cite{GraSwi97};
 \item the set of $a$-values corresponding to case (II) is non-empty \cite{Fei78,Fei79}, and has zero Lebesgue measure \cite{Lyu02};
 \item the set of $a$-values corresponding to case (III) has positive Lebesgue measure 
 \cite{BenCar85,Jak81}. 
  \end{itemize} 
Irrespective of rich bifurcations, Theorem~A states that the LDP continues to hold for an arbitrary parameter $a$.

The rate function $I$ in Theorem~A is given as follows.  
Let $\mathcal M(f)$ denote the set of elements of $\mathcal M$ which are $f$-invariant. 
 For each $\mu\in\mathcal M(f)$, the limit
\[
\chi(x)=
\lim_{n\to\infty}\frac{1}{n}\log|(f^n)'(x)|
\]
exists for $\mu$-almost every $x\in X$ and 
belongs to $\mathbb{R}\cup \{-\infty\}$, because $\log |f'|$ is uniformly bounded from above. We set 
\[\chi^+(x)=\max\{\chi(x),0\}\in \mathbb{R},\]
and define a (non-negative) {\it Lyapunov exponent} of $\mu$ by
\begin{equation}\label{d-lyap}
\chi^+(\mu)=\int \chi^+(x) d\mu(x). 
\end{equation}

Let $h(\mu)$ denote the measure-theoretic entropy of $\mu$ with respect to $f$. Define a {\it free energy} $F\colon \mathcal M\to[-\infty,0]$ by
\begin{equation}\label{eq:F}
F(\mu)=\begin{cases}h(\mu)-\chi^+(\mu)\ &\text{if 
      $  \mu\in\mathcal M(f)$,}\\
      -\infty\ &\text{otherwise.}\end{cases}
\end{equation}
The rate function $I$ in Theorem~A is defined to be the minus of the upper semicontinuous regularization of $F$: 
\[I(\mu)=-\inf_{\mathcal G\ni\mu}\sup_{\mathcal G}F(\nu).\]
Here, the infimum is taken over all open sets $\mathcal G$ in $\mathcal M$ containing $\mu$.
Note that the entropy and the Lyapunov exponent are upper semicontinuous as functions of measures.
 In cases (I) and (II), the Lyapunov exponent is in fact continuous (see 
 Lemma~\ref{rate}) and thus
$I=-F.$
In case (III), 
 the Lyapunov exponent may fail to be lower semicontinuous \cite[$\S$2]{BruKel98}, which implies
 $I\neq -F$.

The next theorem asserts that Theorem~A cannot be extended to maps with multiple critical points.

\begin{theoremb}[The breakdown of the LDP]
There exists a $C^3$ interval map with exactly two non-degenerate critical points for which the level-2 LDP does not hold.
\end{theoremb}

The rest of this paper consists of four sections.   
In $\S$2 we introduce several basic definitions and results on $S$-unimodal maps. 
In $\S$3 we prove the lower bound for open sets in Theorem~A. 
In $\S$4 we prove the upper bound for closed sets in Theorem~A. 
In $\S$5
we prove Theorem~B.

For only finitely renormalizable maps, the dynamics of typical orbits 
consist of two stages: transition to the deepest  
renormalization cycle,
and circulation within that cycle.  
The LDP restricted to the deepest renormalization cycle mostly 
follows from the known result \cite{CRT}.
For infinitely renormalizable maps, orbits contained in sufficiently deep cycles are 
approximated by the post-critical measure supported on the attracting Cantor set.
Therefore, all we have to do is to analyze
the transitions between renormalization cycles.
The main point of $\S$3 is to prove that any $f$-invariant measure supported on a hyperbolic set 
is approximated with orbit segments 
from finitely many subintervals of $X$ (see Proposition~\ref{low}).
This claim is standard if the measure is ergodic, and if non-ergodic then we glue orbit segments together which
approximate its ergodic components.

  The main technique used in $\S$4 to derive the upper bound is a `coarse graining approach'.
Here the set of orbits with prescribed time averages of continuous functions are coarse grained
(see $\S$\ref{coarse}), and estimates on the resultant `clusters'
are transferred to the large deviations upper bound on empirical measures (see Propositions~\ref{up} and \ref{over-lem}).
The estimate on each cluster consists of contributions from the uniformly hyperbolic dynamics
on each renormalization cycle and the transitions between them.

A counterexample of a bimodal map we construct in the proof of Theorem~B is non-transitive and 
has two non-degenerate critical points, one of which is non-recurrent.
With a slight modification of our construction, one can find
a counterexample in polynomial maps of degree $3$,
see Remark~\ref{remark-final}.
It is plausible that there is an $S$-unimodal map with a non-recurrent flat critical point for which the level-2 LDP fails. For a transitive $S$-unimodal map  with a non-recurrent flat critical point and all periodic orbits hyperbolic repelling, the level-2 LDP was shown in
 \cite{ChuTak} under an assumption that the criticality increases at some specific rates.

\section{Preliminaries}
This section introduces several basic definitions and results on the dynamics of 
$S$-unimodal maps and its renormalization.
In $\S$\ref{attractor} we classify attracting periodic orbits. In $\S$\ref{affine-sec} we prove a lemma on the Lyapunov exponent of invariant measures.
In $\S$\ref{metric} 
we define renormalization 
of $S$-unimodal maps. 
In $\S$\ref{g-str} we introduce basic structures associated with the renormalization. In
$\S$\ref{symbolic-dynamics}, $\S$\ref{dist-est},  $\S$\ref{top-section} we are concerned with the dynamics on each renormalization cycle.
In $\S$\ref{notation-sec} we introduce some notation which will be frequently used later.
In $\S$\ref{analysis} we summarize a few results on infinitely renormalizable maps.

\subsection{Classification of attracting periodic orbits}\label{attractor}
Let $f\colon X\to X$ be a unimodal map, and let
$\{f^k(x)\}_{k=0}^{p-1}$
be a periodic orbit of $f$ with prime period $p$.
 We say
 $\{f^k(x)\}_{k=0}^{p-1}$ is:

 \begin{itemize}
\item {\it hyperbolic attracting} if $|(f^p)'(x)|<1$;
 
\item {\it neutral} if $|(f^p)'(x)|=1$; 
 
\item {\it hyperbolic repelling} if $|(f^p)'(x)|>1$.
 \end{itemize}
 The {\it basin} of the periodic orbit $\{f^k(x)\}_{k=0}^{p-1}$ is the set of points in $X$ whose omega-limit set is contained in the set $\{f^k(x)\}_{k=0}^{p-1}$.
We say $\{f^k(x)\}_{k=0}^{p-1}$ is 
{\it attracting} if its basin contains an open set. In this case, the union of the
 connected components of the basin containing a point from $\{f^k(x)\}_{k=0}^{p-1}$
 is called the {\it immediate basin}
 of $\{f^k(x)\}_{k=0}^{p-1}$.
 Each connected component of the immediate basin
 contains exactly one point from $\{f^k(x)\}_{k=0}^{p-1}$.

 If  $f\colon X\to X$ is $S$-unimodal, a neutral periodic point  $x$ with prime period $p$ satisfies 
 either
 \begin{itemize}
     \item[(a)] $(f^p)''(x)\neq 0$, or 
     \item[(b)] $(f^p)''(x)= 0$ and $(f^p)'''(x)/ (f^p)'(x)<0$.
\end{itemize}
  In case (a), the periodic point $x$ is locally attracting
  from only one side  and in case (b)  it is attracting from both sides. 
  Hence a periodic point is attracting if and only if it is hyperbolic attracting or neutral. The immediate basin of an attracting periodic orbit 
contains the critical point 
\cite[Chapter~II, Lemma~6.1, Theorem~6.1]{dMevSt93}. Therefore, there is at most one attracting periodic orbit,
denoted by $O(f)$. 
Let $B(f)$ denote the immediate basin of $O(f)$. If $f$ has no attracting periodic orbit, we put $B(f)=\emptyset$ for convenience.
Let $\delta_{O(f)}$ denote the element of $\mathcal M(f)$ that is supported on $O(f)$.  
The attracting periodic orbit $O(f)$ is called {\it two-sided attracting}
 if it is hyperbolic attracting or neutral in case (b).
 Otherwise ({\it i.e.} neutral in case (a)),  an attracting periodic orbit $\{f^k(x)\}_{k=0}^{p-1}$ is called {\it one-sided attracting}.

\subsection{Lyapunov exponents}\label{affine-sec}
Since our definition of the Lyapunov exponent in \eqref{d-lyap} is a little non-standard, we will use the next lemma in $\S$3.
\begin{lemma}\label{lyap-affine}
Let $f\colon X\to X$ be an $S$-unimodal map with a non-flat critical point. 
 Then 
the map $\mu\in \mathcal M(f)\mapsto\chi^+(\mu)$ is affine. Moreover the following hold:
\begin{itemize}
\item[(a)] if $f$ has no hyperbolic attracting periodic orbit, then $\chi^+(\mu)=\int\log|f'|d\mu$ for any $\mu\in\mathcal M(f)$;
\item[(b)] if $f$ has a hyperbolic attracting periodic orbit, then $\chi^+(\mu)=\int\log|f'|d\mu$ for any $\mu\in\mathcal M(f)$ satisfying $\mu(O(f))=0.$
\end{itemize}
\end{lemma}
\begin{proof} 
Item (a) follows from
 \cite[Proposition~A.1]{Riv20}. Hence, $\chi^+(\cdot)$ is affine if $f$ has no hyperbolic attracting periodic orbit.
Item (b) follows from Ma\~n\'e's theorem \cite[Theorem~A]{Man85} which asserts that the complement of $O(f)$ is a hyperbolic set. Then
 $\chi^+(\cdot)$ is affine on the subspace of measures in $\mathcal M(f)$ not charging $O(f)$. Since every element of $\mathcal M(f)$ can be written uniquely as a convex combination of $\delta_{O(f)}$ and a measure not charging $O(f)$,  $\chi^+(\cdot)$ is also affine in this case.
\end{proof}

\subsection{Renormalization of $S$-unimodal maps}\label{metric}
Let $f\colon X\to X$ be a unimodal map with a non-flat critical point $c$.
   A proper closed subinterval $J$ of $X$ is {\it restrictive} of {\it period} $p\geq2$ if the following hold (c.f. \cite[p.139]{dMevSt93}):
\begin{itemize}
\item the interiors of $J,\ldots,f^{p-1}(J)$ are pairwise disjoint;
\item $f^p(J)\subset J$ and $f^p(\partial J)\subset\partial J$;
\item one of the intervals $J,\ldots,f^{p-1}(J)$ contains $c$ in its interior;
\item $J$ is maximal with respect to these properties: if $J'\supset J$ is a closed interval which is strictly contained in $X$ and satisfies 
the previous three properties with the same integer $p$, then $J'=J$.
\end{itemize}
We define a strictly decreasing sequence of closed intervals 
\[
X=J_0\supsetneq J_1\supsetneq J_2\supsetneq\cdots
\]
which contain the critical point $c$, and a strictly increasing sequence of integers
\[
1=p_0<p_1<p_2<\cdots 
\]
so that $J_m$ is restrictive of period $p_m$ for $m\ge 1$, inductively as follows.
Given $J_m$ and $p_m$ for some $m\geq0$,
 then note that $f^{p_{m}}|_{J_m}\colon J_m\to J_m$ is a unimodal map. 
 If $f^{p_{m}}|_{J_m}$ has a restrictive interval, then define $J_{m+1}$ to be the restrictive interval
of $f^{p_{m}}|_{J_m}$ containing $c$ and with the smallest period $r_m$.
We define $p_{m+1}=p_{m}r_m$. 
If $f^{p_{m}}|_{J_m}$ has no restrictive interval, then we do not define $J_{m+1}$ and stop the definition setting 
$\bar m(f)=m$. 
If this inductive definition continues for arbitrarily large $m$, we set $\bar m(f)=\infty$.
We say $f$ is: 

\begin{itemize}
\item {\it non-renormalizable} if $\bar m(f)=0$;

\item {\it renormarizable} if $\bar m(f)\geq1$;

\item  {\it only finitely renormalizable} if $\bar m(f)<\infty$;

\item {\it infinitely renormalizable}  if $\bar m(f)=\infty$.
\end{itemize}
We assume $f$ has negative Schwarzian derivative, and
 review the dynamics of the unimodal map $f^{p_m}|_{J_m}$, $m\geq0$.
  We classify them into the following cases:
\begin{itemize}
    \item[(A)] $f^{p_m}|_{J_m}$ is non-renormalizable and has no attracting fixed point; 
    \item[(B)] $f^{p_m}|_{J_m}$ is renormalizable; 
    \item[(C)] $f^{p_m}|_{J_m}$ is non-renormalizable and has an attracting fixed point. 
    \end{itemize}
For each $m\geq0$, let $L_m$ denote 
    the closed interval in $J_m$
    bordered by $f^{p_m}(c)$ and $f^{2p_m}(c)$. Note that $L_m$ is possibly a singleton, i.e., a degenerate closed interval.

    In case (A), we have $f^{p_m}(L_m)= L_m$ and  $f^{p_m}|_{L_m}$ is topologically exact: for any relatively open subset $U$ of $L_{m}$ there is an integer multiple $n\geq1$ of $p_m$ such that $f^{n}(U)=L_{m}$.
  One can check this by combining  \cite[Theorem~V.1.3]{dMevSt93} and \cite[Theorem~2.19 and Proposition~2.34]{Rue17} for example.
  
    In case (B), we have 
    $c\in{\rm int}(L_m)$ and $f^{p_m}(L_m)\subset L_m$, and the orbit of any point in ${\rm int}(J_m)$ eventually falls into $L_m$. There exists a restrictive interval $J_{m+1}$ that contains the critical point $c$ and the set $\bigcup_{k=0}^{p_{m+1}/p_m}f^{kp_m}(J_{m+1})$ is forward invariant with respect to $f^{p_m}$. 
    We will consider two subcases
    \begin{itemize}
        \item[(B1)] $p_{m+1}/p_m\neq 2$,  
        or 
        \item[(B2)] $p_{m+1}/p_m=2$.
    \end{itemize}
    Compare these in \textsc{Figure}~\ref{fig:renormalization}.
In case (B2) 
we have $L_{m}=f^{p_m}(J_{m+1})\cup f^{2p_m}(J_{m+1})$ and $L_m\subset \bigcup_{k=0}^{p_{m+1}-1} f^k(J_{m+1})$.

    \begin{figure}
    \begin{overpic}[scale=0.7]{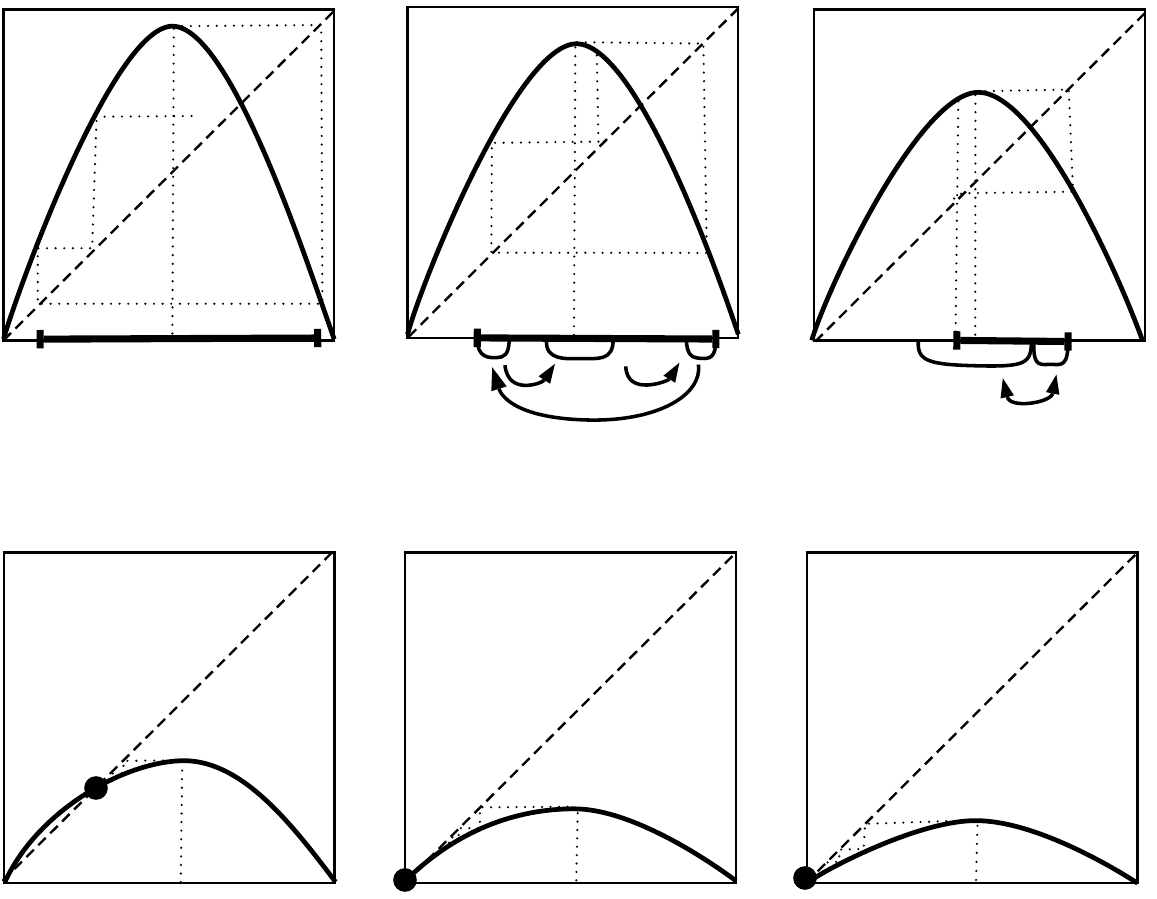}
    \put(11,37){(A)}
    \put(46,37){(B1)}
    \put(83,37){(B2)}
    \put(11,-7){(C1)}
    \put(46,-7){(C2)}
    \put(83,-7){(C3)}
    \put(16,50){$L_m$}
    \put(51,50){$L_m$}
    \put(85,50){$L_m$}
    \put(48,43){$J_{m+1}$}
    \put(8,6){$z_f$}
     \put(35,-2){$z_f$}
      \put(70,-2){$z_f$}
    \put(80,42){$J_{m+1}$}
    \end{overpic}
    \vspace{2cm}
    \caption{The graphs of the renormalized unimodal maps $f^{p_{m}}|_{J_{m}}\colon J_m\to J_m$.}
    \label{fig:renormalization}
\end{figure}
   
    In case (C), 
    $m=\bar m(f)$ and $J_m$ contains a unique point $z_f$ from $O(f)$. Indeed, 
 if $J_m$ contained two points
    from $O(f)$, then one would find a restrictive interval for $f^{p_m}|_{J_m}$ 
    as an immediate basin of $O(f)$. 
 In particular, the point $z_f\in O(f)\cap J_m$ is a fixed point of the unimodal map $f^{p_m}|_{J_m}$. 
 We will consider three subcases:
\begin{itemize}
    \item[(C1)]  $z_f\in{\rm int}(J_m)$;
    \item[(C2)] $z_f\in\partial J_{m}$ and
 $O(f)$ is neutral;
 \item[(C3)] 
  $z_f\in\partial J_{m}$ and $O(f)$ is hyperbolic attracting. 
\end{itemize}
Compare these in \textsc{Figure}~\ref{fig:renormalization}.
     In case (C1), $O(f)$ is hyperbolic attracting and
 ${\rm int}(J_{m})$ is the connected component of $B(f)$ containing $z_f$.
 For any closed interval $J$ contained in ${\rm int}(J_{m})$,
 the following uniform convergence holds:
 \begin{equation}\label{unic}\displaystyle{\lim_{n\to\infty}\sup_{x\in J}}|f^{p_{m}n}(x)-z_f|=0.\end{equation}
  In cases (C2) and (C3),
 since $f$ has negative Schwarzian derivative
 $f^{p_{m}}(c)$ lies in between $z_f$ and $c$ and the following uniform convergence on $J_m$ holds:
  \begin{equation}\label{unic1}\displaystyle{\lim_{n\to\infty}\sup_{x\in J_{m}}}|f^{p_{m}n}(x)-z_f|=0.\end{equation}
If moreover $z_f$ is two-sided attracting and $m\geq1$, the dynamics of the previous renormalization $f^{p_{m-1}}|_{J_{m-1}}$ is similar to that in case (C1) from Lemma~\ref{stop} below. For this reason, if $1\leq \bar m(f)<\infty$ and $\partial J_{\bar m(f)}$ contains a two-sided
attracting fixed point of $f^{p_{\bar m(f)}}|_{J_{\bar m(f)}}$, then
we zoom out to the previous renormalization
by setting
\[
m(f)=\bar m(f)-1.
\]
In all other cases, we set 
\[
m(f)=\bar m(f).
\]

\begin{lemma}\label{stop}
Let $f\colon X\to X$ be a renormalizable $S$-unimodal map, and let $J$ be a restrictive interval with period $p$ 
containing $c$ and not contained in any other restrictive interval with period smaller than $p$. If $z_f\in\partial J$ is periodic and two-sided attracting, then $z_f$ is a fixed point of $f$, $p=2$ and its immediate basin coincides with ${\rm int}(X)$. Further, for any closed interval $J$ contained in $\mathrm{int}(X)$, we have the uniform convergence  \eqref{unic} on $J$ with $p_m=1$. 
\end{lemma}
\begin{proof}
Let $B$ denote the connected component of the immediate basin of the periodic orbit of $z_f$ that contains $z_f$. Then  we have $f^p(B)\subset B$.
Since the renormalized map $f^p|_J$ belongs to case (C2) or (C3) and since $z_f$ is two-sided attracting, the interval $J$ is strictly contained in $B$. 

We claim that $f(B)\subset B$. Suppose that this is not the case. Then, by the definition of immediate basin, we have that $f(B)\cap B= \emptyset$. This implies that ${\rm cl}(B)$ is a restrictive interval and 
contradicts the assumption of the lemma.

Since the immediate basin $B$ contains only one point in the orbit of $z_f$, the point $z_f$ is a fixed point of $f$. 
Then it follows that $p=2$. Since $f(\partial B)\subset \partial B$, it is immediate to see that $B$ coincides with ${\rm int}(X)$.   
\end{proof}

\begin{remark}\label{rem:C}
To summarize, in the case $m(f)<\infty$ the dynamics of the  renormalization $f^{p_{m(f)}}|_{J_{m(f)}}$ is either in case (A), or one of the following:
\begin{itemize}
    \item[(C-I)] there is a two-sided attracting fixed point in ${\rm int}(J_{m(f)})$ whose immediate basin contains ${\rm int}(J_{m(f)})$;
    \item[(C-II)] there is a one-sided attracting fixed point in $\partial J_{m(f)}$ whose immediate basin contains  $J_{m(f)}$.
\end{itemize}
\end{remark}

\subsection{Structures associated with renormalization}\label{g-str}
Let $f\colon X\to X$ be an 
$S$-unimodal map with a non-flat critical point $c$ such that $m(f)\geq1$.
For each integer $m$ with $0\leq m< m(f)$, 
 we define the $m$-th {\it renormalization cycle} $K_m$ by  
\[
K_{m}=\bigcup_{k=0}^{p_{m}-1} f^k(J_{m}).
\]
 For each integer $m$ with $0\leq m<m(f)-1$, we denote by
$\tilde{\mathscr{P}}_{m}$ the collection of the connected components of 
$K_m\setminus \bigcup_{k=0}^{p_{m+1}-1}{\rm int}(f^k(J_{m+1}))$.
 
 In the case $m(f)<\infty$, we set
\begin{equation}\label{KMF}K_{m(f)}=\begin{cases}
\bigcup_{k=0}^{p_{m(f)}-1} f^k(J_{m(f)}) &\text{ if $f$ has an attracting periodic orbit},\\
\bigcup_{k=0}^{p_{m(f)}-1} f^k(L_{m(f)}) &\text{ if $f$ has no attracting periodic orbit.}\end{cases}
\end{equation}
 Moreover,
we define $\tilde{\mathscr{P}}_{m(f)-1}$ to be the collection of the connected components of the following sets:
$K_{m(f)-1}\setminus \bigcup_{k=0}^{p_{m(f)}-1}{\rm int}(f^k(J_{m(f)}))$
if $f$ has an attracting periodic orbit;
$K_{m(f)-1}\setminus \bigcup_{k=0}^{p_{m(f)}-1}{\rm int}(f^k(L_{m(f)}))$
if $f$ has no attracting periodic orbit.

The elements of $\tilde{\mathscr{P}}_m$ with $0\leq m<m(f)$ are closed intervals in $K_m$, possibly singletons, 
and $\tilde{\mathscr{P}}_m$ has the Markov property: if $P,Q\in\tilde{\mathscr{P}}_m$ and $f(P)\cap Q\neq\emptyset$ 
then $f(P)\supset Q$. 
Let $\mathscr{P}_m$ denote the set of elements of $\tilde{\mathscr{P}}_m$ that are contained in $\bigcup_{k=0}^{p_m-1}f^k(L_m)$.
  Note that $\mathscr{P}_m$ is non-empty.
  
 \begin{remark}
Case~(B2) is rather exceptional. The set
  $\mathscr{P}_m$ consists of the singletons 
  $f^{k}(J_{m+1})\cap f^{p_m+k}(J_{m+1})$, $0\leq k\leq p_m-1$, 
  which form a hyperbolic repelling periodic orbit of prime period $p_m$. 
 \end{remark}
 
 For each integer $m$ with $0\leq m< m(f)$,
 we define 
 \[K_{m,m+1}={\rm cl}(K_m\setminus K_{m+1}),\]
 and \[
\Gamma_{m}= \bigcap_{n=0}^\infty f^{-n}\left(\bigcup_{P\in \mathscr{P}_m}P\right).\]
Note that $\Gamma_m$ for $0\leq m< m(f)$ contains an attracting periodic orbit only if $m=m(f)-1$,
and this periodic orbit is one-sided attracting.
 In this case, $K_{m,m+1}$
is disjoint from the interior of the immediate basin of the attracting periodic orbit.
For convenience we set $\Gamma_{-1}=\partial X.$ 
In the case $m(f)<\infty$ 
we further define 
\[\Gamma_{m(f)}=\begin{cases}  
O(f)&\quad\text{if $f$ has a two-sided attracting periodic orbit, 
}\\
\bigcup_{k=0}^{p_{m(f)}-1}f^k(L_{m(f)})&\quad\text{otherwise.}
\end{cases}\]
The sets $\Gamma_m$, $-1\leq m\leq m(f)$ are non-empty closed sets, and $\Gamma_m\cap\Gamma_{m'}=\emptyset$
holds for all distinct integers $m,m'$ in $\{-1,\ldots,m(f)-1\}$.
Note that $\Gamma_{m(f)-1}$ intersects $\Gamma_{m(f)}$ if and only if 
$L_{m(f)}=J_{m(f)}$.

For each integer $m$ with $-1\leq m\leq m(f)$ we set
\[\mathcal M_m(f)=\{\mu\in\mathcal M(f)\colon{\rm supp}(\mu)\subset \Gamma_m\},\]
where ${\rm supp}(\mu)$ denotes the smallest closed subset of $X$ with full $\mu$-measure.
 The sets  $\mathcal M_m(f)$ are pairwise disjoint,
   and if $m(f)<\infty$ then $\mathcal M(f)$ is the convex hull of  $\bigcup_{m=-1}^{m(f)}\mathcal M_m(f)$.
If $m(f)<\infty$ and $f$ has a one-sided attracting periodic orbit, then
$L_{m(f)}$ is contained in the basin of this attracting periodic orbit, and so $\mathcal M_{m(f)}(f)=\emptyset$.
The set $\mathcal M_{-1}(f)$ is a singleton that is supported on the fixed point in $\partial X$.

\subsection{Symbolic dynamics on each cycle}\label{symbolic-dynamics}
Let $f\colon X\to X$ be an $S$-unimodal map such that $m(f)\geq1$.
  For each integer $m$ with $0\leq m< m(f)$,
there is a topological Markov chain over the finite alphabet $\tilde{\mathscr{P}}_m$  determined by the transition matrix 
\[
(M_{PQ})_{P,Q\in\tilde{\mathscr{P}}_m},\quad
M_{PQ}=\begin{cases} 1& \text{if $f(P)\supset Q$,}\\
0& \text{otherwise}.
\end{cases}
\]  
Let $n\geq2$ be an integer 
 and let $P_0,P_1,\ldots,P_{n-1}\in \tilde{\mathscr{P}}_m$. 
 The word $P_0P_1\cdots P_{n-1}$ of length $n$ is {\it admissible} if 
 $M_{P_kP_{k+1}}=1$ holds for $0\leq k\leq n-2$.
 Let
 $E^n_m$ denote the set of admissible words of elements of $\tilde{\mathscr{P}}_m$
 of length $n$.
Let $\Sigma_m$ denote the set of one-sided infinite 
sequences $\{P_k\}_{k=0}^\infty$ of elements of $\tilde{\mathscr{P}}_m$
such that $P_0\cdots P_{n-1}\in E_m^n$ holds for every $n\geq2$. 
We endow $\Sigma_m$ with 
the restriction of the product topology of the discrete topology of $\tilde{\mathscr{P}}_m$.
Let $\sigma_m\colon\Sigma_m\to\Sigma_m$ denote the left shift: $\sigma_m(\{P_k\}_{k=0}^\infty)=\{P_k\}_{k=1}^\infty$.
  For each $P_0P_1\cdots P_{n-1}\in E^n_m$, we set
$I_{P_0P_1\cdots P_{n-1}}=\bigcap_{k=0}^{n-1}f^{-k}(P_k)$,
and define
\[\pi_m\colon \Sigma_m\to \bigcap_{n=0}^{\infty} f^{-n}\left(\bigcup_{P\in \tilde{\mathscr{P}}_m}P\right)\] by
\begin{equation}\label{pim}\pi_m(\{P_k\}_{k=0}^\infty)\in\bigcap_{n=1}^\infty I_{P_0\cdots P_{n-1} }.\end{equation}

By \cite[Main~Theorem]{Lyu89}, $f$ has no wandering interval, and so any homterval is contained in the basin of an attracting periodic orbit.
Since $\bigcap_{n=1}^\infty I_{P_0\cdots P_{n-1}}$
is not contained in the basin of an attracting periodic orbit, it is not a homterval, namely, a singleton.
Hence, $\pi_m$ is well-defined.

\begin{prop}\label{sft}
Let $f\colon X\to X$ be an $S$-unimodal map with a non-flat critical point such that $m(f)\geq1$.
For each integer $m$ with 
$0\leq m< m(f)$, 
the restriction 
 of $f$ to
$\bigcap_{n=0}^{\infty} f^{-n}\left(\bigcup_{P\in \tilde{\mathscr{P}}_m}P\right)$
is topologically conjugate by the conjugacy map $\pi_m$ to the topological Markov chain $\sigma_m\colon\Sigma_m\to\Sigma_m$.
\end{prop}
\begin{proof}
By definition, $\pi_m$ is continuous and surjective. Since the elements of $\tilde{\mathscr{P}}_m$ are pairwise disjoint, $\pi_m$ is injective and has a continuous inverse. 
 For each $P_0P_1\cdots P_{n-1}\in E^n_{m}$, 
the Markov property of $\tilde{\mathscr{P}}_m$ implies
$f^{k-1}(I_{P_0\cdots P_{n-1}})=I_{P_{k-1}\cdots P_{n-1}}$ for every $1\leq k\leq n$.
Hence $f\circ\pi_m=\pi_m\circ\sigma_m$ holds.
We have verified that 
the restriction of $f$
 to
$\bigcap_{n=0}^{\infty} f^{-n}\left(\bigcup_{P\in \tilde{\mathscr{P}}_m}P\right)$
is topologically conjugate
to $\sigma_m$ by $\pi_m$.
\end{proof}

We will use the following lemma a few times.
\begin{lemma}\label{u-decay}
 Let $f\colon X\to X$ be an $S$-unimodal map with a non-flat critical point such that $m(f)\geq1$.
For each integer $m$ with
$0\leq m< m(f)$, 
 we have \[
 \lim_{n\to\infty}\sup\{
 \left|I_{\omega}\right|\colon   \omega\in E^n_{m}\}=0.\]
\end{lemma}
\begin{proof}
 Suppose there exist $\varepsilon>0$ and an infinite subset $F$ of 
   $\bigcup_{n=1}^\infty E^n_{m}$ such that
  $|I_{\omega}|>\varepsilon$ for all $\omega\in F$.
  Since $X$ has a finite diameter, 
  we can choose a sequence $\{\omega^{(n)}\}_{n=1}^\infty$ in $F$
  such that   $\{I_{\omega^{(n)}}\}_{n=1}^\infty$ is a nested sequence.
  Since $\bigcap_{n=1}^\infty I_{\omega^{(n)}}$ is not a singleton, 
    it is a homterval. However,
  it is not contained in the basin of an attracting periodic orbit. We obtain a contradiction to \cite[Main Theorem]{Lyu89}.
  \end{proof}

\subsection{Distortion estimate on each cycle}\label{dist-est}
Let $Y$ be a non-degenerate compact interval in $X$ and let $g\colon Y\to X$ be a $C^1$ map.
Let $J$ be a subinterval of $Y$ such that the restriction of $g$ to $J$ is a diffeomorphism onto its image. By a {\it distortion of $g$ on $J$} we mean the quantity
\[\sup_{x,y\in J}\frac{|g'(x)|}{|g'(y)|}.
\]
We will frequently use the following distortion estimates on each cycle.

\begin{prop}\label{koebe}
Let $f\colon X\to X$ be an $S$-unimodal map with a non-flat critical point such that $m(f)\geq1$. For each integer $m$ with
$0\leq m< m(f)$ the following hold:
\begin{itemize}
\item[(a)] 
if $\Gamma_m$ does not contain a neutral periodic orbit, then
there exists a constant $\gamma_m\geq1$ such that
if $n\geq1$ then 
the distortion of $f^n$ on any
connected component of $\bigcap_{k=0}^{n-1}f^{-k}(K_{m,m+1})$ is bounded by $\gamma_m$;

    \item[(b)]if
    $\Gamma_{m}$ contains a neutral periodic orbit, then for any $\varepsilon>0$ there exists $N\geq1$ such that 
    if $n\geq N$ then the distortion of $f^n$ on every connected component of 
    $\bigcap_{k=0}^{n-1}f^{-k}(K_{m,m+1})$ is bounded by $e^{\varepsilon n}$. 
    \end{itemize}
\end{prop}

\begin{proof}
Since $f$ is $C^2$, the restriction of $\log|f'|$ to
the union of the elements of $\tilde{\mathscr{P}}_m$ 
not containing a point
from an attracting periodic orbit
 is Lipschitz continuous. Moreover, by Ma\~n\'e's theorem \cite[Theorem~A]{Man85},
 the maximal $f$-invariant set in this union is a hyperbolic set.
 Hence (a) holds.

As for (b),
if $\Gamma_m$ contains a neutral periodic orbit, then
$m=m(f)-1$ and $\partial J_{m(f)}$ contains a point from the neutral periodic orbit. 
  Let $\varepsilon>0$, let $n\geq2$ and let
  $W$ be a connected component of $\bigcap_{k=0}^{n-1}f^{-k}(K_{m,m+1})$.
 There exists
  $\omega\in E^n_{m}$ such that
  $W=I_\omega$.
    Since $K_{m,m+1}$ does not contain the critical point of $f$, the infimum of $|f'|$  over this set is positive.
For all $x,y\in I_\omega$
we have
\[\begin{split}\log\frac{|(f^n)'(x)|}{|(f^n)'(y)|}&\leq \left(\sup_{K_{m,m+1} }\frac{|f''|}{|f'|}\right)\sum_{k=0}^{n-1}|f^k(I_\omega)|\leq\left(\sup_{K_{m,m+1} }\frac{|f''|}{|f'|}\right) \sum_{k=0}^{n-1}\sup_{\omega\in E^{n-k}_m}|I_{\omega}|.\end{split}\]
By Lemma~\ref{u-decay}, the last number is bounded by $\varepsilon n$ for a sufficiently large $n$. 
The proof of Proposition~\ref{koebe} is complete.
\end{proof}

\subsection{Topological exactness on each cycle}\label{top-section}
 Let $Y$ be a non-degenerate interval in $X$ and $g\colon Y\to Y$ be a continuous map.
Let $U$ be a non-degenerate subinterval of $Y$ and let $n\in\mathbb N$. Each connected component of $g^{-n}(U)$
  is called {\it a pullback of $U$ by $g^n$}. 
  If $V$ is a pullback of $U$ by $g^n$ and 
  $g^n|_V$ is a diffeomorphism, then $V$ is called {\it a diffeomorphic pullback of $U$ by $g^n$.}

The next lemma will be used later 
in $\S$\ref{k} for the proof of the lower bound in Theorem~A, in order to glue different orbits in differenc cycles together to form one orbit with required properties.
 It will also be used later in $\S$\ref{escape}
for the proof of the upper bound
in Theorem~A.

\begin{lemma}\label{lem-Lm}
Let $f\colon X\to X$ be an $S$-unimodal map with a non-flat critical point such that $m(f)\geq1$.
Let $0\leq m<m(f)$ and suppose $p_{m+1}/p_m\neq 2$.
For any open interval 
$J\subset L_m$ that intersects $\Gamma_m$, there exists an integer $n\geq1$ such that $f^{p_m n}(J)=L_m$, and for any $n'\ge n$ there exists a diffeomorphic pullback of ${\rm int}(L_m)$ by $f^{p_m n'}$ that is contained in $J$.
\end{lemma}
\begin{proof}
We assume $m=0$ in the proof below. In the case $m>0$, we can apply the same argument to the renormalized map $f^{p_m}\colon J_m \to J_m$ to obtain the conclusion of the lemma. 

Without loss of generality we may assume $X=[0,1]$ and the critical point $c$ is the maxima of $f$.
We claim that the topological entropy $h_{\rm top}(f)$ of $f\colon X\to X$ is greater than or equal to 
$\log\sqrt{2}$. Indeed, the assumption $p_1=p_1/p_0>2$ implies that 
if we take an interval $[x',x]$ such that $x'<c<x$ and $f(x)=x=f(x')$, then 
 $f^2(x)=x$, $f^2(c)<x'$ and $[x',x]\subset f^2([x',x])$. Since $[x',x]$ contains an $f^2$-invariant set on which $f^2$ is topologically conjugate to the full shift on two symbols, we obtain $h_{\rm top}(f)= h_{\rm top}(f^2)/2\ge \log\sqrt{2}$.

By the Milnor-Thurston kneading theory (see \cite[Ch.2~Theorem~8.1]{dMevSt93}), there exists a semi-conjugacy from $f\colon X\to X$ to a piecewise linear map with slopes $\pm s$ where $s=e^{h_{\rm top}(f)}\ge \sqrt{2}$.  More precisely, there exists a continuous and surjective monotone map $\lambda\colon X\to X$ satisfying $T\circ \lambda=\lambda\circ f$ for the tent map 
\[
T\colon X\to X, \quad T(x)=\min\{sx,s(1-x)\}.
\]
The map $\lambda$ sends a non-trivial interval in $X$ to a point only if it is eventually mapped by $f$ into the restrictive interval $J_1$. (For this, we use the non-existence of wandering intervals. See \cite[III.4, Proposition 4.3]{dMevSt93}.)
Hence $\lambda$ maps an open interval $J\subset L_0$ that intersects the Cantor set 
$\Gamma_0$ onto a non-trivial interval $J'$. 

If $s=\sqrt{2}$, then $T$ is renormalizable with period $2$ and so is $f$, contradicting the assumption $p_1>2$. 
So we may suppose $s> \sqrt{2}$. Then there is $n\ge 0$ such that both $T^n(J')$ and $T^{n+1}(J')$ contain the turning point $1/2$ of $T$ in their interiors (see \cite[III.4,~Proposition~4.4 and Exercise~4.1]{dMevSt93}). Consequently, $f^n(J)$ and $f^{n+1}(J)$ contain $c$ in their interiors and therefore  $f^{n+2}(J)\supset L_0=[f^{2}(c),f(c)]$. Then we have $f^{n'}(J)\supset f^{n'-n-2}(L_0)=L_0$ for $n'\ge n+2$.

Since $f$ is unimodal, for any non-degenerate subinterval $I$ of $X$ we can take a subinterval $I'\subset I$ such that $f\colon I' \to f(I)$ is a homeomorphism.  Applying this procedure to the sequence of maps
\[
\begin{CD}
    J@>{f}>> f(J) @>{f}>> \cdots@>{f}>>  f^{n'-1}(J) @>{f}>> f^{n'}(J)\supset L_0,
\end{CD}
\]
we find an interval $J'\subset J$ such that $f^{n'}\colon J' \to L_0$ is a homeomorphism. Then ${\rm int}(J')$ is a diffeomorphic pullback of ${\rm int}(L_0)$ by $f^{n'}$, as required in the lemma.
\end{proof}

Let $f\colon X\to X$ be an $S$-unimodal map with a non-flat critical point such that $m(f)\geq1$.  Suppose $p_{m+1}/p_m\neq 2$ for some $0\leq m<m(f)$.
In view of Lemma~\ref{lem-Lm}, 
for each $P\in\mathscr{P}_m$ that is contained in $L_m$ we 
fix $q(P)\in\mathbb N$ 
such that $f^{p_m q(P)}(P)=L_m$, and for any $n\ge q(P)$ there exists a diffeomorphic pullback of ${\rm int}(L_m)$ by $f^{p_m n}$ that is contained in $P$. 
For later uses and convenience we set
\begin{equation}\label{Mm-def}M_m=\max\{p_mq(P)\colon P\in\mathscr{P}_m\text{ and }P\subset L_m\},\ \text{ and }\ M_{-1}=2.\end{equation}

\subsection{Notation}\label{notation-sec}
Let $C(X)$ denote the set of real-valued continuous functions on $X$.
For $\phi\in C(X)$ and $n\in\mathbb N$, write $S_n\phi$ for the sum
$\sum_{k=0}^{n-1}\phi\circ f^k$.
For $\ell\in\mathbb N$ define
\[
 C(X)^\ell=\{\vec{\phi}=(\phi_1,\ldots,\phi_\ell)\colon\phi_j\in C(X)\ \text{ for }1\leq j\leq\ell\}.
\]
For $\vec{\phi}=(\phi_1,\ldots,\phi_\ell)\in C(X)^{\ell}$, $\vec{\alpha}=(\alpha_1,\ldots,\alpha_\ell)\in\mathbb R^{\ell}$
and a measure $\mu\in\mathcal M$, the expression 
$\int\vec{\phi} d\mu>\vec{\alpha}$ indicates that $\int\phi_j d\mu>\alpha_j$ holds for $1\leq j\leq\ell$. 
We set 
\[
A_n(\vec{\phi},\vec{\alpha})=
\left\{x\in X\colon \int\vec\phi d\delta_x^n>\vec\alpha\right\},
\]
and \[S_n\vec{\phi}=(S_n\phi_1,\ldots,S_n\phi_\ell).\]
For $a\in\mathbb R$ and $\vec v=(v_1,\ldots,v_\ell)\in\mathbb R^{\ell}$, we write 
\[
\vec{a}=(a,a,\ldots,a)\in\mathbb R^{\ell}\ \text{ and }\ \|\vec{v}\|=\max_{1\leq j\leq\ell}|v_j|.
\]
For $\vec v,\vec w\in\mathbb R^{\ell}$, the expression $\vec v\geq \vec w$
indicates that $v_j\geq w_j$ holds for $1\leq j\leq\ell$.

\subsection{Infinitely renormalizable maps}\label{analysis}
Let $f\colon X\to X$ be an infinitely renormalizable $S$-unimodal map with a non-flat critical point. 
The omega limit set of Lebesgue almost every initial point $x\in X$ coincides with the closed invariant set
\[
\Lambda=\bigcap_{m=0}^\infty K_m
\]
and the empirical measure $\delta_x^n$ converges to {\it the post-critical measure} $\mu_\infty$ supported on $\Lambda$. The restriction of $f$ to $\Lambda$ is uniquely ergodic and the post-critical measure $\mu_\infty$ is the unique invariant measure on $\Lambda$. Moreover, $h(\mu_\infty)=\chi^+(\mu_\infty)=0$ holds (see 
\cite[Theorem~3.4(b)]{BruKel98}).

In order to `approximate' $\mu_\infty$, we will use the following lemma, which in particular
 asserts that the post-critical measure is approximated
 by empirical measures along sufficiently long orbit segments in sufficiently deep cycles.

\begin{lemma}
  [{\cite[Lemma~2.3]{Tak21}}]\label{deep}
 Let $f\colon X\to X$ be an $S$-unimodal map with a non-flat critical point such that $m(f)=\infty$.
 Let 
 $\ell\in\mathbb N$, $\vec{\phi}\in C(X)^{ \ell}$, $\vec{\alpha}\in\mathbb R^{\ell}$.
 For any $\varepsilon>0$  there exist $m_*$, $N_*\in\mathbb N$ 
 such that the following hold:
 \begin{itemize}
\item[(a)] if $n\geq N_*$ and
  $ A_{n}(\vec{\phi},\vec{\alpha})\cap K_{m_*}\neq\emptyset$
  then 
\[\int\vec\phi d\mu_\infty>\vec\alpha-\vec\varepsilon;\]
\item[(b)]
if $n\geq N_*$ then
\[\sup_{K_{m_*}}\left\|\frac{1}{n}S_n\vec\phi-\int\vec\phi d\mu_\infty\right\|<\varepsilon.\]
\end{itemize}
  \end{lemma}
 For infinitely renormalizable maps,
the Lyapunov exponent depends continuously on invariant measures.
\begin{lemma}[{\cite[Lemma~2.6]{Tak21}}]\label{rate}
Let $f\colon X\to X$ be an $S$-unimodal map with a non-flat critical point such that $m(f)=\infty$. Then
$\mu\in\mathcal M(f)\mapsto\chi^+(\mu)$ is continuous. 
\end{lemma}

\section{Large deviations lower bound}
In this section we complete the proof of the lower bound in Theorem~A.
In $\S$\ref{keylow} we show that this follows from a key lower bound stated in Proposition~\ref{low}.
The rest of this section is dedicated to the proof of Proposition~\ref{low}.

\subsection{Key lower bound}\label{keylow}

The next proposition allows us to approximate a given invariant measure with finitely many intervals
in a particular sense.
\begin{prop}\label{low}
Let $f\colon X\to X$ be an $S$-unimodal map with a non-flat critical point.
Let $\ell\in\mathbb N$, $\vec{\phi}\in C(X)^{ \ell}$, $\vec{\alpha}\in\mathbb R^{\ell}$ and
let $\mu\in\mathcal M(f)$ satisfy $\int\vec\phi d\mu>\vec{\alpha}$.
 For any $\varepsilon>0$ there exists $N\geq1$ such that
 for each integer $n\geq N$, there exists
a finite collection $\mathscr{A}_{n}$ of closed intervals in
$A_{n}(\vec{\phi},\vec{\alpha})$ with pairwise disjoint interiors such that 
\begin{align}\tag{a}\#\mathscr{A}_{n}&>\exp\left((h(\mu)-\varepsilon)n\right)\quad and\\
\tag{b}
 |A|&>\exp\left(-(\chi^+(\mu)+\varepsilon)n\right)\quad\text{for every }A\in\mathscr{A}_{n}.
\end{align}
\end{prop}
In order to prove Proposition~\ref{low}, we   approximate the 
given measure $\mu$ 
with a finite convex combination of measures each supported on some $\Gamma_m$ or $\Lambda$, and then for each of these measures construct a collection of finitely many intervals. We then glue orbits from these intervals together to construct a collection intervals with the desired properties.

The proof of Proposition~\ref{low} is given in $\S$\ref{endow}.
Below we deduce a corollary to Proposition~\ref{low} and complete the proof of
the lower bound in 
Theorem~A.
\begin{cor}\label{des}Let $f\colon X\to X$ be an $S$-unimodal map with a non-flat critical point.
Let $\ell\in\mathbb N$, $\vec{\phi}\in C(X)^{\ell}$, $\vec{\alpha}\in\mathbb R^{\ell}$ and let $\mu\in\mathcal M(f)$ satisfy $\int\vec{\phi} d\mu>\vec{\alpha}$. Then
\[
\liminf_{n\to\infty}\frac{1}{n}\log |A_n(\vec{\phi},\vec{\alpha})|\geq  F(\mu).
\]\end{cor}
\begin{proof}
 From Proposition~\ref{low} and \eqref{eq:F}, for any $\varepsilon>0$ there exists $N\geq1$ such that
$|A_{n}(\vec{\phi},\vec{\alpha})|>
 e^{(F(\mu)-2\varepsilon)n}$ holds for every $n\geq N$.
 Taking logarithms of both sides, dividing by $n$ and letting $n\to \infty$ and then 
 $\varepsilon\to 0$, we obtain the desired inequality.
\end{proof}

\begin{proof}[Proof of the lower bound in Theorem~A]
Note that subsets of $\mathcal M$ of the form
$\{\mu\in\mathcal M\colon\int\vec\phi d\mu>\vec\alpha\}$
with $\ell\in\mathbb N$, $\vec{\phi}\in C(X)^{ \ell}$, $\vec{\alpha}\in\mathbb R^{\ell}$
constitute a base of the weak* topology of $\mathcal M$.
So, any non-empty open subset
 $\mathcal G$ of $\mathcal M$ is written as the union 
$\mathcal G=\bigcup_{\lambda}\mathcal G_\lambda$ of sets $\mathcal G_\lambda$ of this form. 
For each $\mathcal G_\lambda$, Corollary~\ref{des} gives
\[
\liminf_{n\to\infty}\frac{1}{n}\log|\{x\in X\colon\delta_x^n\in\mathcal G_\lambda\}|
\geq \sup_{\mathcal G_\lambda}F.
\] 
Hence we obtain
\[\liminf_{n\to\infty}\frac{1}{n}\log| \{x\in X\colon\delta_x^n\in\mathcal G\}|
\geq\sup_{\lambda}\sup_{\mathcal G_\lambda} F= \sup_{\mathcal{G}} F= -\inf_{\mathcal G} I,
\]
as required in Theorem~A.
\end{proof}

 \subsection{Approximation of a measure on  each cycle}\label{k} The rest of this section is entirely dedicated to a proof of Proposition~\ref{low}.
We fix $\ell\in\mathbb N$, $\vec{\phi}\in C(X)^{\ell}$,  $\vec{\alpha}\in\mathbb R^{\ell}$ in the statement of Proposition~\ref{low} till the end of $\S$\ref{endow},
and assume $f\colon X\to X$ is an $S$-unimodal map with a non-flat critical point $c$.
 
The next lemma allows us to approximate an invariant measure on each single cycle with a finite collection of intervals.
\begin{lemma}\label{katok}
Let $f\colon X\to X$ satisfy $m(f)\geq1$. 
  Let $0\leq m< m(f)$
  and suppose $p_{m+1}/p_m\neq 2$. Let
  $\mu\in\mathcal M_m(f)$, $P\in\mathscr{P}_m$ satisfy $P\subset L_m$ and
  $\mu(P)>0$, and let $P'$ be a non-degenerate closed interval in $L_{m}$.
  For any $\varepsilon>0$, there exists $N\geq1$ such that for each  integer $q\geq N$, there exists a finite collection $\mathscr{B}_{p_mq}(P,P')$ of pullbacks of $P'$ by $f^{p_mq}$ that are contained in $P$ such that 
 \begin{align}\tag{a}\left|\frac{1}{p_mq}\log\# \mathscr{B}_{p_mq}(P,P')-h(\mu)\right|&<\varepsilon\quad and\\
\tag{b}
\sup_{B\in\mathscr{B}_{p_mq}(P,P')}\sup_{B}\left\|\frac{1}{p_mq}S_{p_mq}\vec\phi-\int\vec\phi d\mu\right\|&<\varepsilon.
\end{align}
If moreover $P'\cap\{f(c),\ldots,f^{M_m-1}(c)\}=\emptyset$ for the integer
 $M_m\geq1$ in \eqref{Mm-def}, 
then every pullback $B\in\mathscr{B}_{p_mq}(P,P')$ is diffeomorphic and satisfies
\begin{equation}\tag{c}
 \sup_{B}\left|\frac{1}{p_mq}\log|(f^{p_mq})'|-\chi^+(\mu)\right|<\varepsilon.
\end{equation}
 \end{lemma}
\begin{proof} 
From the definition of $\pi_m$ in \eqref{pim} and Proposition~\ref{sft},  $f|_{\Gamma_m}$ is topologically conjugate to a topological Markov chain   via the finite Markov partition $\mathscr{P}_m$. For $q\in\mathbb N$ let $\tilde{\mathscr{B}}_{p_m q}(P)$ denote the collection of the pullbacks of elements of $\mathscr{P}_m$ by $f^{p_m q}$ that are contained in $P$. Lemma~\ref{u-decay} gives $\lim_{q\to\infty}\sup_{B\in\tilde{\mathscr{B}}_{p_mq}(P)}|B|=0$, and so for any continuous function $\varphi\colon K_{m,m+1}\to\mathbb R$ we have \begin{equation}\label{w-decay-c}\lim_{q\to\infty}\sup_{B\in\tilde{\mathscr{B}}_{p_m q}(P)}\sup_{x,y\in B}\frac{1}{p_m q}\left|S_{p_m q}\varphi(x)-S_{p_m q}\varphi(y)\right|=0.\end{equation}

The proof of Lemma~\ref{katok} breaks into two cases. We first treat the case where $\mu$ is ergodic. 
Since $m<m(f)$, $\mu$ is not supported on the orbit of a hyperbolic attracting periodic orbit (see $\S$\ref{g-str}).
Let $\varepsilon>0$. 
Since $\mu(L_m)=1/p_m$, the normalized restriction of $\mu$ to $L_m$ is an $f^{p_m}|_{L_m}$-invariant Borel probability measure
with entropy $p_mh(\mu)$. 
For $q\in\mathbb N$, we set $r=q-M_m /p_m$.
Since $P\in\mathscr{P}_m$,
    $\tilde{\mathscr{B}}_{p_m r}(P)$
    coincides with the collection of the pullbacks of elements of $\mathscr{P}_m$ by $f^{p_m r}$ that are contained in $P$. 
We claim that for all sufficiently large $q$ there exists a subset $\mathscr{B}_{p_m r}(P)$ of $\tilde{\mathscr{B}}_{p_m r}(P)$ for which the following hold:
\begin{align*}\left|\frac{1}{p_m r}\log\#\mathscr{B}_{p_m r}(P)-h(\mu)\right|&<\frac{\varepsilon}{3};\\
\sup_{B\in\mathscr{B}_{p_m r}(P)}\sup_{B}\left\|\frac{1}{p_m r}S_{p_m r}\vec\phi-\int \vec\phi d\mu\right\|&<\frac{\varepsilon}{3};\\
\sup_{B\in\mathscr{B}_{p_m r}(P)}\sup_{B}\left|\frac{1}{p_m r}\log|(f^{p_m r})'|-\chi^+(\mu)\right|&<\frac{\varepsilon}{3}.\end{align*}
Indeed, the first inequality follows from Shannon-McMillan-Breiman's theorem \cite{CFS82} for
the normalized restriction. The second one follows from 
\eqref{w-decay-c} and
Birkhoff's ergodic theorem, and
the third one follows from Proposition~\ref{koebe},
 Birkhoff's ergodic theorem and $\chi^+(\mu)=\int\log|f'|d\mu$ by Lemma~\ref{lyap-affine}.

By the definition of $M_m$ in \eqref{Mm-def}, if $Q\in \mathscr{P}_m$ and $Q\subset L_m$ then 
$f^{M_m}(Q)=L_m$.
For each $ B\in \mathscr{B}_{p_m r}(P)$, the interval $f^{p_m r}(B)$ is an element of $\mathscr{P}_m$ that is contained in $L_m$.
 We fix a connected component $W_B$ of $f^{p_m r}(B)\cap f^{-M_m}(P')$.
 Let $B'$ denote the diffeomorphic pullback of $W_B$ by $f^{p_m r}$ that is contained in $B$. We define
 \[\mathscr{B}_{p_mq}(P,P')=\{B'\colon B\in\mathscr{B}_{p_mr}(P)\}.\]
 Clearly we have $\#\mathscr{B}_{p_mq}(P,P')=\#\mathscr{B}_{p_m r}(P)$, and
 the elements of $\mathscr{B}_{p_mq}(P,P')$ are
 pullbacks of $P'$ by $f^{p_m q}$. 
  If $P'\cap\{f(c),\ldots,f^{M_m-1}(c)\}=\emptyset$
then for every $ B\in \mathscr{B}_{p_m r}(P)$,
 $W_B$ is a diffeomorphic pullback of $P'$ by $f^{M_m}$. So, every
  element of $\mathscr{B}_{p_mq}(P,P')$ is a diffeomorphic
 pullback of $P'$ by $f^{p_m q}$. 
 Therefore, for all sufficiently large $q\geq1$ we obtain the inequalities in (a) (b) (c) of Lemma~\ref{katok} from the corresponding three estimates for $\mathscr{B}_{p_m r}(P)$ as above. 
To deduce (c) we have noted that the orbits of  points in elements of $\mathscr{B}_{p_mq}(P,P')$ up to time $p_mq$ are uniformly bounded away from the critical point $c$ by a distance independent of $q$. 

It is left to treat the case where $\mu$ is non-ergodic. 
 By Lemma~\ref{lyap-affine}, the map $\nu\in \mathcal{M}_m(f)\mapsto \chi^+(\nu)\in \mathbb{R}$ is affine.
By virtue of the ergodic decomposition theorem and Jacobs' theorem on the decomposition of entropy \cite{Jak60}, for any $\varepsilon>0$ there exist a finite number of ergodic measures $\mu_1,\ldots,\mu_s$
in $\mathcal M_m(f)$ and constants $\rho_1,\ldots,\rho_s$ in $(0,1)$ adding up to $1$ 
such that the measure $\mu'=\sum_{i=1}^s\rho_i\cdot\mu_i$ in $\mathcal M_m(f)$ satisfies 
\[
|h(\mu)-h(\mu')|<\varepsilon, \quad 
\left\|\int\vec{\phi}d\mu-\int\vec{\phi}d\mu'\right\|<\varepsilon,\quad |\chi^+(\mu)-\chi^+(\mu')|<\varepsilon,
\]
and $\mu'(P)$ is sufficiently close to $\mu(P)$ so that $\mu'(P)>0$.
 With no loss of generality
we may assume that $\mu$ is written as a convex combination of ergodic measures $\mu_i\in \mathcal M_m(f)$:
\[
\mu=\rho_1\cdot\mu_1+\rho_2\cdot\mu_2+\cdots+\rho_s\cdot\mu_s, 
\]
with $\mu_1(P)>0$ and $\rho_1,\ldots,\rho_s$ adding up to $1$.
 We fix $P_1,\ldots,P_s\in \mathscr{P}_m$ so that $P_1=P$, $P_i\subset L_m$ and $\mu_i(P_i)>0$ for  $1\leq i\leq s$. We set $P_{s+1}=P'$.

Let $\varepsilon>0$. 
Let $q\ge s$ be a large integer, and write it as a sum of positive integers \[q= q_1+q_2+\cdots+q_s,\]
with $|p_mq_i-p_m\rho_i q|\leq p_m$
for $1\leq i\leq s.$
From the previous argument in the ergodic case,
if $q$ is sufficiently large then
for each $i$ 
there is a finite collection $\mathscr{B}_{p_mq_i}(P_i,P_{i+1})$ of diffeomorphic pullbacks of $P_{i+1}$ by $f^{p_m q_i}$ that is 
contained in $P_i$ for which the following three inequalities hold:
 \begin{align}\label{katok1}
 \left|\frac{1}{p_m q_i}\log\# \mathscr{B}_{p_mq_i}(P_i,P_{i+1})-h(\mu_i)\right|&<\frac{\varepsilon}{2};\\
\label{katok2} \sup_{B\in\mathscr{B}_{p_mq_i}(P_i,P_{i+1})  }\sup_{B}
\left\|\frac{1}{p_m q_i}S_{p_m q_i}\vec\phi-\int\vec\phi  d\mu_i\right\|&<
\frac{\varepsilon}{2};\\
\label{katok3} \sup_{B\in\mathscr{B}_{p_mq_i}(P_i,P_{i+1})   }\sup_{B}
\left|\frac{1}{p_m q_i}\log|(f^{p_m q_i})'|-\chi^+(\mu_i)\right|&<
\frac{\varepsilon}{2}.
\end{align}
 We now define $\mathscr{B}_{p_mq}(P,P')$ to be the collection of the diffeomorphic pullbacks of $P'$ by the composition
$g_s\circ  g_{s-1}\circ\cdots \circ  g_1$
of diffeomorphisms  $g_i$ of the form $f^{p_m q_i}|_{B}$ with $B\in\mathscr{B}_{p_mq_i}(P_i,P_{i+1})$. If $q$ is sufficiently large, then by \eqref{katok1} we have 
\[\begin{split}
\left|\frac{1}{p_m q}\log\#\mathscr{B}_{p_mq}(P,P')-h(\mu) \right|\leq&\frac{1}{p_m q}\left|
\sum_{i=1}^s\log\#\mathscr{B}_{p_mq_i}(P_i,P_{i+1})   -p_m q \sum_{i=1}^s\rho_i h(\mu_i)        \right|\\
\leq& \frac{1}{p_m q}\sum_{i=1}^s
\left|\log\#\mathscr{B}_{p_mq_i}(P_i,P_{i+1})  -p_m q_i h(\mu_i)  \right|\\
&+ \sum_{i=1}^s\left|\frac{q_i}{q}-\rho_i\right| h(\mu_i) 
<\frac{\varepsilon}{2}+\frac{s}{q}h_{\rm top}(f) <\varepsilon,
\end{split}\]
where $h_{\rm top}(f)$ denotes the topological entropy of $f$.
Moreover, for each $B\in\mathscr{B}_{p_mq}(P,P')$,
 using \eqref{katok2} for $1\leq i\leq s$ 
 we have
\[\begin{split}\sup_B\left\|\frac{1}{p_m q}S_{p_m q}\vec\phi-  \int\vec\phi d\mu\right\|
\leq&\frac{1}{p_m q}\sum_{i=1}^s \sup_{B'\in\mathscr{B}_{p_mq_i}(P_i,P_{i+1}) }\sup_{B'}\left\|S_{p_m q_i}\vec\phi-p_m q_i\int\vec\phi d\mu_i\right\|\\
&+\sum_{i=1}^s \left|\frac{q_i}{q}-\rho_i\right| \|\vec\phi \|
<\frac{\varepsilon}{2}+\frac{s}{q}\|\vec \phi\| < \varepsilon.
\end{split}\]
 Since $\varepsilon>0$ is arbitrary, the above two estimates imply (a) and (b) in Lemma~\ref{katok} respectively. 
We can prove (c) in the same way as (b), 
using \eqref{katok3} 
and noting that the orbits of  points in elements of $\mathscr{B}_{p_mq}(P,P')$ up to time $p_mq$ are uniformly bounded away from the critical point $c$ by a distance independent of $q$. 
 This completes the proof of Lemma~\ref{katok}.
\end{proof}

\subsection{The initial transition} Let $\mu_{-1}$ denote the element of the singleton
$\mathcal M_{-1}(f)$. Recall that $\mu_{-1}$ is the unit point mass at the fixed point in $\partial X$ (see $\S$\ref{g-str}).  Put $p_{-1}=1$. Recall that $L_0$ is the closed interval in $X$
    bordered by $f(c)$ and $f^{2}(c)$ (see $\S$\ref{metric}). Let $x_f$ denote the fixed point of $f$ in $\partial X$. If $L_0$ is strictly contained in $X$, let $P_{-1}$ denote the connected component of $X\setminus L_0$ 
 that contains $x_f$.

 \begin{lemma}\label{katok-initial}
Let $f\colon X\to X$ be $S$-unimodal such that $L_0$ is strictly contained in $X$
and is not a singleton. 
Let $P$ be a non-degenerate closed interval in $\mathrm{int}(X)$.  Suppose either
\begin{itemize}
    \item[(i)]  
    $f$ is renormalizable and $P\subset L_0$, or
    \item[(ii)] 
    $f$ is non-renormalizable and 
    $P$ is contained in the open interval bordered by $x_f$ and the unique fixed point $z_f$ of $f$ in $\mathrm{int}(X)$ (see (C1) in \textsc{Figure~1}).
\end{itemize}
For any $\varepsilon>0$, there exists $N\geq1$ such that for each  integer $n\geq N$, there exists a pullback $B$ of $P$ by $f^{n}$ that is contained in $P_{-1}$ such that 
\begin{equation}\tag{a}
\sup_{B}\left\|\frac{1}{n}S_{n}\vec\phi-\int\vec\phi d\mu_{-1}\right\|<\varepsilon.
\end{equation}
If moreover $f(c)\notin P$, 
then the pullback $B$ is diffeomorphic and satisfies
\begin{equation}\tag{b}
 \sup_{B}\left|\frac{1}{n}\log|(f^{n})'|-\chi^+(\mu_{-1})\right|<\varepsilon.
\end{equation}
 \end{lemma}
 
 \begin{proof}
 In case (i),
 $f(c)$ is not contained in the closed interval bordered by $x_f$ and $c$.
 Without loss of generality we may assume $x_f<c$. 
 We define a sequence $\{x_n\}_{n=2}^\infty$ in $(x_f,c)$ inductively by
  $x_2\in f^{-1}(c)$  and
 $x_{n+1}\in f^{-1}(x_n)$ for every $n\geq2$.
 From the minimum principle for negative Schwarzian unimodal maps \cite{dMevSt93}, this sequence is well-defined and satisfies $\lim_{n\to\infty}x_n= x_f$. The interval $[x_f,x_n]$ is mapped by $f^{n}$ homeomorphically onto $[x_f,f(c)]$. Let $B$ denote the pullback of $P$ by $f^{n}$ that is contained in $[x_f, x_n]$. If $f(c)\notin P$, then $B$ is a diffeomorphic pullback. If $n$ is sufficiently large, then (a) (b) follow from the bounded distortion near $x_f$. 
 
A proof for case (ii) is analogous with $c$ replaced by $z_f$. The proof of Lemma~\ref{katok-initial} is complete. \end{proof}

\subsection{Proof of Proposition~\ref{low}}\label{endow}
If $L_0=X$ then $f$ is non-renormalizable, and the conclusion of Proposition~\ref{low} was already shown in the proof of \cite[Proposition~2.1]{CRT}.
If $L_0$ is a singleton, then $f(c)=c$ and any element of $\mathcal M(f)$ is a convex combination of $\mu_{-1}$ and the unit point mass at $c$. Hence, it is not difficult to show the conclusion of Proposition~\ref{low} by slightly modifying the argument in Case~2 below.

We assume $L_0$ is strictly contained in $X$ and not a singleton.
Let $\mu\in\mathcal M(f)$ satisfy $\int\vec{\phi} d\mu>\vec\alpha$.
Let $\varepsilon>0$ be small enough so that
\begin{equation}\label{low-new}\int\vec{\phi}d\mu-\vec\varepsilon>\vec{\alpha}.\end{equation}
  In what follows we treat four cases separately.
We start with the case where $f$ is infinitely renormalizable. Arguments for the remaining cases will proceed much in parallel to the infinitely renormalizable case.
\medskip

\noindent{\bf Case 1:} $m(f)=\infty$.
Let $m_*$, $N_*$ be positive integers for which the conclusion of Lemma~\ref{deep} holds with $\varepsilon$ replaced by $\varepsilon/3$. 
By the continuity of the Lyapunov exponent in Lemma~\ref{rate},
 we may assume with no loss of generality that
there exist 
constants $\rho_{-1},\ldots,\rho_{m_*}\in(0,1)$ that add up to $1$, 
and for each $m\in\{-1,\ldots,m_*-1\}$
a measure $\mu_m\in\mathcal M_{m}(f)$ such that 
\[\mu=\rho_{-1}\cdot\mu_{-1}+\rho_{0}\cdot\mu_{0}+\cdots+\rho_{m_*-1}\cdot\mu_{m_*-1}+\rho_{m_*} \cdot\mu_\infty.
\]
For each $m\in\{0,\ldots,m_*-1\}$ we take a non-degenerate closed interval $P_m$ so that the following hold:
if $p_{m+1}/p_{m}\neq 2$ then  $P_m\in\mathscr{P}_{m}$, $P_m\subset L_{m}$ and $\mu_m(P_m)>0$; 
if 
 $p_{m+1}/p_{m}= 2$ then  
 $P_m\subset L_{m}$, $P_m$ contains the periodic point in $\partial J_{m+1}$
 in its interior and satisfies
 $P_{m}\cap\{f(c),\ldots,f^{M_{m-1}-1}(c)\}=\emptyset$ (recall \eqref{Mm-def}).
 We set $P_{m_*}=J_{m_*}$.

For each $m\in\{-1,\ldots, m_*-1\}$ and for sufficiently large $q\in\mathbb N$, we define a collection
$\mathscr{B}_{p_{m}q}(P_m,P_{m+1})$ of non-degenerate closed subintervals of $P_m$ as follows:

\begin{itemize}
\item $m=-1$:
$\mathscr{B}_{p_{-1}q}(P_{-1},P_{0})$ is the singleton consisting of the pullback of $P_0$ 
by $f^{p_{-1}q}$ that is obtained by applying Lemma~\ref{katok-initial} with $P=P_0$. Since $f(c)\notin P_0$ by the choice of $P_0$, this pullback is diffeomorphic by Lemma~\ref{katok-initial}.

\item $m\geq0$:
if $p_{m+1}/p_{m}\neq 2$ then 
 $\mathscr{B}_{p_{m}q}(P_m,P_{m+1})$ is the collection of the diffeomorphic pullbacks of $P_{m+1}$ by $f^{p_{m}q}$ contained in $P_m$ that are obtained by applying Lemma~\ref{katok} with $\mu=\mu_m$, $P=P_m$, $P'=P_{m+1}$ and $\varepsilon$ replaced by $\varepsilon/2$;
if $p_{m+1}/p_{m}=2$ then 
$\mathscr{B}_{p_{m}q}(P_m,P_{m+1})$ is the singleton consisting of the diffeomorphic pullback of
$P_{m+1}$ by $f^{p_{m}q}$ that is contained in $P_m$ and closest to the periodic point in $\partial J_{m+1}$ among all such diffeomorphic pullbacks.
\end{itemize}

We claim that
for every
$-1\leq m\leq m_*-1$ the following hold provided $q$ is sufficiently large:
\begin{align}\label{tag1}
\left|\frac{1}{p_{m}q}\log\# \mathscr{B}_{p_{m}q}(P_m,P_{m+1})-h(\mu_m)\right|&<\frac{\varepsilon}{2};\\
\label{tag2}
\sup_{B\in\mathscr{B}_{p_{m}q}(P_m,P_{m+1})}\sup_{B}\left\|\frac{1}{p_{m}q}S_{p_{m}q}\vec\phi-\int
\vec\phi  d\mu_m\right\|&<\frac{\varepsilon}{2};\\
\label{tag3}
\sup_{B\in\mathscr{B}_{p_{m}q}(P_m,P_{m+1})}\sup_{B}\left|\frac{1}{p_{m}q}\log|(f^{p_{m}q})'|-\chi^+(\mu_m)\right|&<\frac{\varepsilon}{2}.
\end{align}
Indeed,
for $m=-1$,
 \eqref{tag1} is trivial since $\mu_{-1}$ is the unit point mass at the fixed point of $f$ in $\partial X$, and \eqref{tag2}, \eqref{tag3} are consequences of Lemma~\ref{katok-initial}.
Let $m\geq0$. If 
 $p_{m+1}/p_{m}=2$ then 
 $\mu_m$ is the empirical measure on the orbit of the hyperbolic repelling periodic point in  $\partial J_{m+1}$.
 Therefore, \eqref{tag1}, \eqref{tag2} and \eqref{tag3} hold in these two cases.
If $p_{m+1}/p_{m}\neq 2$ then \eqref{tag1}, \eqref{tag2} and \eqref{tag3} 
 are consequences of Lemma~\ref{katok}.

We write $n\in\mathbb N$ as a linear combination
of non-negative integers $q_{-1},\ldots,q_{m_*-1},r$ in the form
\[
n= p_{-1}q_{-1}+p_0q_0+\cdots+p_{m_*-1}q_{m_*-1}+r,\]
with 
$|p_{m}q_m- \rho_m n|\le  1$ for $-1\leq m\leq m_*-1$
and $|r- \rho_{m_*} n|\le  1.$
In what follows we assume $n$ is sufficiently large so that $r\geq N_*$, 
and for each $-1\leq m\leq m_*-1$,
$\mathscr{B}_{p_{m}q_m}(P_m,P_{m+1})$ satisfies \eqref{tag1}, \eqref{tag2}, \eqref{tag3}.
Lemma~\ref{deep}(b) gives
\begin{equation}\label{eq-low-new}\sup_{K_{m_*}}\left\|\frac{1}{r}S_r\vec\phi-\int\vec\phi d\mu_\infty\right\|<\frac{\varepsilon}{2}.\end{equation}

Let $\mathscr{A}_n$ denote the collection of the pullbacks of $P_{m_*}$ by maps
$h_{m_*-1}\circ\cdots \circ  h_{-1}$,
where each $h_m$ 
is a diffeomorphism of the form $f^{p_{m}q_m}|_{B}$ with $B\in\mathscr{B}_{p_{m}q_m}(P_m,P_{m+1})$. 
By $\#\mathscr{A}_n=\prod_{m=-1}^{m_*-1}\#\mathscr{B}_{p_{m}q_m}(P_m,P_{m+1})$, \eqref{tag1} and $h(\mu_{\infty})=0$, we have
\begin{equation}\label{entropy}
\left| \frac{1}{n}\log\#\mathscr{A}_n-h(\mu) \right|<\varepsilon.
\end{equation}
A similar argument to the proofs of (b) and (c) in Lemma~\ref{katok} based on
 \eqref{tag2} and \eqref{eq-low-new} shows
\begin{equation}\label{eq:tima_average_estiamte}
\sup_{A\in\mathscr{A}_n}\sup_A\left\| \frac{1}{n}S_{n}\vec\phi-\int\vec\phi  d\mu \right\|<\varepsilon.
\end{equation}
From $\chi^+(\mu_\infty)=0$, 
$\chi^+(\mu)=\sum_{m=-1}^{m_*-1}\rho_m
\chi^+(\mu_m)$ by Lemma~\ref{lyap-affine} and \eqref{tag3}, we obtain
\begin{equation}\label{eq:derivative_estimate}
\sup_{A\in\mathscr{A}_n}\sup_A\left| \frac{1}{n-r}\log|(f^{n-r})'|-\chi^+(\mu) \right|<\frac{\varepsilon}{2}.
\end{equation}
The inequality in Proposition~\ref{low}(a) follows from \eqref{entropy}.
 From \eqref{low-new} and \eqref{eq:tima_average_estiamte}, the elements of $\mathscr{A}_n$ are contained in $A_{n}(\vec{\phi},\vec{\alpha})$. The inequality in Proposition~\ref{low}(b)
 follows from \eqref{eq:derivative_estimate}. 
  \medskip

 \noindent{\bf Case 2:} {\it $m(f)<\infty$ and ${\rm int}(J_{m(f)})$ contains a two-sided attracting  periodic point (Remark~\ref{rem:C}(C-I)).}
 Recall that
  $\delta_{O(f)}$ denotes the element of $\mathcal M(f)$ supported on $O(f)$. With no loss of generality
 we may assume
 there exist 
constants $\rho_{-1},\ldots,\rho_{m(f)}\in(0,1)$ adding up to $1$ 
and for each $m\in\{-1,\ldots,m(f)-1\}$ a measure $\mu_m\in\mathcal M_{m}(f)$ 
such that 
\[\mu=\rho_{-1}\cdot\mu_{-1}+\rho_0\cdot\mu_0+
\cdots+\rho_{m(f)}\cdot\delta_{O(f)}.
\]
For each  $m\in\{0,\ldots,m(f)\}$ we take a non-degenerate closed interval $P_m$ so that the following hold: if $m<m(f)$ and $p_{m+1}/p_{m}\neq 2$ then $P_m\in\mathscr{P}_{m}$,
$P_m\subset L_{m}$ and $\mu_m(P_m)>0$; 
if $m<m(f)$ and $p_{m+1}/p_{m}= 2$ then 
$P_m\subset L_m$, $P_m$  contains the periodic point in $\partial J_{m+1}$
 in its interior and satisfies
 $P_{m}\cap\{f(c),\ldots,f^{M_{m-1}-1}(c)\}=\emptyset$;
 $P_{m(f)}$ is contained in the connected component of the immediate basin of the attracting periodic orbit that contains $c$.
 In the case of $m(f)=0$, we take $P_{m(f)}$ to be a closed interval contained in
 the open interval bordered by $x_f$ and the unique fixed point in $z_f$.

For each $m\in\{-1,\ldots, m(f)-1\}$ and for sufficiently large $q\in\mathbb N$, we define a collection
$\mathscr{B}_{p_{m}q}(P_m,P_{m+1})$ of non-degenerate closed subintervals of $P_m$ as follows:

\begin{itemize}
\item $m=-1$: $\mathscr{B}_{p_{-1}q}(P_{-1},P_{0})$ is the singleton consisting of the pullback of $P_0$ 
by $f^{p_{-1}q}$ that is obtained by applying Lemma~\ref{katok-initial} with $P=P_0$. Since $f(c)\notin P_0$ by the choice of $P_0$, this pullback is  diffeomorphic by Lemma~\ref{katok-initial}.

\item $m\geq0$: 
if  $p_{m+1}/p_{m}\neq2$
then $\mathscr{B}_{p_{m}q}(P_m,P_{m+1})$ is the collection of the diffeomorphic pullbacks of $P_{m+1}$ by $f^{p_{m}q}$ contained in $P_m$ that are obtained by applying Lemma~\ref{katok} with $\mu=\mu_m$, $P=P_m$, $P'=P_{m+1}$ and $\varepsilon$ replaced by $\varepsilon/2$; 
if $p_{m+1}/p_{m}=2$ 
then $\mathscr{B}_{p_{m}q}(P_m,P_{m+1})$ is the singleton consisting of the diffeomorphic pullback of $P_{m+1}$ by $f^{p_{m}q}$ that is contained in $P_m$ and closest to the periodic point in $\partial J_{m+1}$ among all such diffeomorphic pullbacks. 
\end{itemize}
Note that the second case $m\geq0$ occurs only if $m(f)\geq1$.

Let $n\geq1$ be an integer, and write it as a linear combination of positive integers $q_1,\ldots,q_{m(f)-1}$ in the form
\[
n= p_{-1}q_{-1}+\cdots+p_{m(f)-1}q_{m(f)-1}+r,\]
with $|p_{m}q_m- \rho_m n|\le 1$ for $-1\leq m\leq m(f)-1$ and
 $|r- \rho_{m(f)} n|\le  1.$ In what follows
 we assume $n$ is sufficiently large so that for each $m\in\{-1,\ldots,m(f)-1\}$,
$\mathscr{B}_{p_{m}q_m}(P_m,P_{m+1})$ satisfies \eqref{tag1}', \eqref{tag2}' and \eqref{tag3}' which are exactly the same as \eqref{tag1}, \eqref{tag2} and \eqref{tag3} respectively under the different assumption on $f$, and 
 \begin{equation}\label{conv-uniform}
\sup_{P_{m(f)}}\left\|\frac{1}{r }S_{r }\vec\phi-\int\vec\phi
d\delta_{O(f)}\right\|<\frac{\varepsilon}{2}.
\end{equation} 

Let $\mathscr{A}_n$ denote the collection of the pullbacks of $P_{m(f)}$ by maps
$h_{m(f)-1}\circ\cdots \circ  h_0\circ h_{-1}$
where $h_m$ for $-1\leq m\leq m(f)-1$ is a diffeomorphism of the form $f^{p_{m}q}|_{B}$ with $B\in\mathscr{B}_{p_{m}q}(P_m,P_{m+1})$. 
From $\#\mathscr{A}_n=\prod_{m=-1}^{m(f)-1}\#\mathscr{B}_{p_{m}q_m}(P_{m},P_{m+1})$ and 
\eqref{tag1}' we obtain Proposition~\ref{low}(a).
    From \eqref{low-new}, \eqref{tag2}' and \eqref{conv-uniform}, the elements of $\mathscr{A}_n$ are contained in $A_{n}(\vec{\phi},\vec{\alpha})$. Since the measure $\delta_{O(f)}$ is supported on $O(f)$,  
     we have $\chi^+(\delta_{O(f)})=0$ from the definition \eqref{d-lyap} of the Lyapunov exponent and $\chi^+(\mu)=\sum_{m=-1}^{m(f)-1}\rho_m\chi^+(\mu_m)$ by Lemma~\ref{lyap-affine}.
    Hence \eqref{tag3}' implies Proposition~\ref{low}(b). 
\medskip

\noindent{\bf Case~3:} {\it $m(f)<\infty$ and $\partial J_{m(f)}$ contains a one-sided attracting  periodic point (Remark~\ref{rem:C}(C-II)).} 
This case is treated by a slight modification of the argument for Case~2.
\medskip
 
\noindent{\bf Case~4:} {\it $m(f)<\infty$ and $f$ has 
 no attracting periodic orbit.} 
 We may assume
 there exist 
constants $\rho_{-1},\ldots,\rho_{m(f)}\in(0,1)$ adding up to $1$, 
and for each $m\in\{-1,\ldots,m(f)\}$ a measure $\mu_m\in\mathcal M_{m}(f)$ 
such that 
\[\mu=\rho_{-1}\cdot\mu_{-1}+\rho_0\cdot\mu_0+\cdots+\rho_{m(f)}\cdot\mu_{m(f)}.
\]

Set $g=f^{p_{m(f)}}|_{L_{m(f)}}$, and
let $\nu$ denote the normalized restriction of $\mu_{m(f)}$ to $L_{m(f)}$.
Let $h(g,\nu)$, $\chi^+(g,\nu)$ denote the entropy and Lyapunov exponent of the measure $\nu$ with respect to $g$. By Lemma~\ref{lyap-affine} we have $\chi^+(g,\nu)=\int\log|Dg|d\nu$.
As in $\S$\ref{metric},
 $g$ is topologically exact.
From the proof of \cite[Proposition~2.1]{CRT}, 
for all sufficiently large $q\geq1$
there exist a closed interval $P_{m(f)}$ in $L_{m(f)}$ and a finite collection $\mathscr{B}_{p_{m(f)}q}(P_{m(f)},P_{m(f)})$
of diffeomorphic pullbacks of $P_{m(f)}$ by $g^q$ that are contained in $P_{m(f)}$ such that
 \begin{equation}\label{tag111}
\begin{split} \#\mathscr{B}_{p_{m(f)}q}(P_{m(f)},P_{m(f)})&>\exp\left(\left( h(g,\nu)-\frac{\varepsilon}{2}\right)q\right)\\
&=\exp\left(\left(p_{m(f)}h(\mu_{m(f)})-\frac{\varepsilon}{2}\right)q\right),\end{split}\end{equation}
 and for every $B\in\mathscr{B}_{p_{m(f)}q}(P_{m(f)},P_{m(f)})$,
   \begin{equation}\label{tag222}\begin{split}\sup_{ B}&\left\|\frac{1}{q}\sum_{k=0}^{q-1}\vec \phi\circ g^k-\int\vec\phi d\nu\right\|=\sup_{ B}\left\|\frac{1}{q}S_{p_{m(f)}q}\vec\phi-p_{m(f)}\int\vec\phi d\mu_{m(f)}\right\|<\frac{\varepsilon}{2},\end{split}\end{equation} 
 and
  \begin{equation}\label{tag333}
  |B|>\exp\left(-\left(\chi^+(g,\nu)+\frac{\varepsilon}{2}\right)q\right)=\exp\left(-\left(p_{m(f)}\chi^+(\mu_{m(f)})+\frac{\varepsilon}{2}\right)q\right).
  \end{equation}

For each $m\in\{0,\ldots, m(f)-1\}$, we take a non-degenerate closed interval $P_m$ such that the following hold:
if $p_{m+1}/p_{m}\neq 2$ then  $P_m\in\mathscr{P}_{m}$,
$P_m\subset L_{m}$ and $\mu_m(P_m)>0$:
if $p_{m+1}/p_{m}= 2$ then
$P_m\subset L_m$, $P_m$ contains the periodic point in $\partial J_{m+1}$
 in its interior and satisfies
 $P_{m}\cap\{f(c),\ldots,f^{M_{m-1}-1}(c)\}=\emptyset$.

 For each $m\in\{-1,\ldots, m(f)-1\}$ and for sufficiently large $q\in\mathbb N$, we define a collection
$\mathscr{B}_{p_{m}q}(P_m,P_{m+1})$ of non-degenerate closed subintervals of $P_m$ as follows:

\begin{itemize}
\item $m=-1$:  $\mathscr{B}_{p_{-1}q}(P_{-1},P_{0})$ is the singleton consisting of the pullback of $P_0$ 
by $f^{p_{-1}q}$ that is obtained by applying Lemma~\ref{katok-initial} with $P=P_0$. Since $f(c)\notin P_0$ by the choice of $P_0$, this pullback is  diffeomorphic by Lemma~\ref{katok-initial}.

\item $m\geq0$: if $p_{m+1}/p_{m}\neq 2$ then  $\mathscr{B}_{p_{m}q}(P_m,P_{m+1})$ is the collection of the  diffeomorphic pullbacks of $P_{m+1}$ by $f^{p_{m}q}$ contained in $P_m$ obtained by applying Lemma~3.3 with $\mu=\mu_m$, $P=P_m$, $P'=P_{m+1}$ and $\varepsilon$ replaced by $\varepsilon/2$;
if $p_{m+1}/p_{m}=2$ then 
 $\mathscr{B}_{p_{m}q}(P_m,P_{m+1})$ is the singleton consisting of the diffeomorphic pullback of $P_{m+1}$ by $f^{p_{m}q}$ that is contained in $P_m$ and closest to the periodic point in $\partial J_{m+1}$ among all such diffeomorphic pullbacks.
\end{itemize}
Note that the second case $m\geq0$ occurs only if $m(f)\geq1$.

Let $n\geq1$ be an integer that is written as a linear combination of non-negative integers $q_{-1},\ldots,q_{m(f)},r$ in the form 
\[n= p_{-1}q_{-1}+\cdots+p_{m(f)-1}q_{m(f)-1}+p_{m(f)}q_{m(f)}+r,\]
with $|p_{m}q_m- \rho_m n|\le  1$ for $-1\leq m\leq m(f)$ and $0\leq r<p_{m(f)}$. We assume $n$ is sufficiently large so that
for each $m$,  
 $\mathscr{B}_{p_{m}q_m}(P_{m},P_{m+1})$ satisfies  
\eqref{tag1}'', \eqref{tag2}'' and \eqref{tag3}'' which are exactly the same as \eqref{tag1}, \eqref{tag2} and \eqref{tag3} respectively under the different assumption on $f$,
 and
 $\mathscr{B}_{p_{m(f)}q_{m(f)}}(P_{m(f)},P_{m(f)})$ satisfying
 \eqref{tag111}, \eqref{tag222} and \eqref{tag333}
 exists.
  We set $P_{m(f)+1}=P_{m(f)}$ for convenience.
 Define $\mathscr{A}_n$ to be the collection of  the pullbacks of $P_{m(f)}$ by 
maps
$ h_{m(f)}\circ\cdots \circ  h_0\circ h_{-1}$,
where each $h_m$, $-1\leq m\leq m(f)$ 
is a diffeomorphism of the form $f^{p_{m} q_m}|_{B}$ with $B\in\mathscr{B}_{p_{m}q_m}(P_{m},P_{m+1})$. From $\#\mathscr{A}_n=\prod_{m=-1}^{m(f)}\#\mathscr{B}_{p_{m}q_m}(P_{m},P_{m+1})$, 
\eqref{tag1}'' and
\eqref{tag111} 
we obtain the inequality in Proposition~\ref{low}(a).
  From \eqref{low-new}, \eqref{tag2}'' and \eqref{tag222}, the elements of $\mathscr{A}_n$ are contained in $A_{n}(\vec{\phi},\vec{\alpha})$. By Lemma~\ref{lyap-affine} we have $\chi^+(\mu)=\sum_{m=-1}^{m(f)}\rho_m\chi^+(\mu_m)$. From \eqref{tag3}'' and
\eqref{tag333} we obtain the inequality in Proposition~\ref{low}(b). 
The proof of Proposition~\ref{low} is complete. 
\qed

\section{Large deviations upper bound}\label{upbound}
In this section we complete the proof of the upper bound in Theorem~A.
In $\S$\ref{key} we show that this follows from a key upper bound stated in Proposition~\ref{up}.
The rest of this section is dedicated to a proof of Proposition~\ref{up}.

\subsection{Key upper bound}\label{key}

Let $\ell\in\mathbb N$,   $\vec{\phi}\in C(X)^{ \ell}$, $\vec{\alpha}\in\mathbb R^{\ell}$.   
Define
\[
\cA_n(\vec{\phi},\vec{\alpha})=
\left\{x\in X\colon \int\vec\phi d\delta_x^n\geq\vec\alpha\right\}.
\] 
\begin{prop}\label{up}
Let $f\colon X\to X$ be an $S$-unimodal map with a non-flat critical point.
Let $\ell\in\mathbb N$, $\vec{\phi}\in C(X)^{ \ell}$, $\vec{\alpha}\in\mathbb R^{\ell}$.
For any $\varepsilon>0$ there exists $N\in\mathbb N$ such that if
$n\geq N$ and $\cA_n(\vec{\phi},\vec{\alpha})\neq\emptyset$, then  there 
exists a measure $\mu\in\mathcal M(f)$ such that
\begin{align}\tag{a}
    |\cA_n(\vec{\phi},\vec{\alpha})|&\leq  \exp\left((F(\mu)+\varepsilon)n\right)\quad and\\
\tag{b}\int\vec\phi  d\mu&>\vec\alpha-\vec\varepsilon.
\end{align}
\end{prop}

We finish the proof of the upper bound in Theorem~A assuming Proposition~\ref{up}.

\begin{proof}[Proof of the upper bound in Theorem~A]
Let $\mathcal C$ be a non-empty closed  
subset of $\mathcal M$.
Let $\mathcal G$ be an arbitrary open set containing $\mathcal C$.
Since $\mathcal M$ is metrizable and $\mathcal C$ is compact, we can choose $\varepsilon>0$ and finitely many closed sets $\mathcal{C}_1,\ldots,\mathcal C_s$ of the form 
$
\mathcal C_i=\{\mu\in\mathcal M\colon \int\vec\phi_i d\mu\geq\vec\alpha_i\}$
with $\ell_i\in\mathbb N$, $\vec\phi_i\in C(X)^{ \ell_i}$, $\vec\alpha_i\in\mathbb R^{\ell_i}$
for $1\leq i\leq s$
so that 
\begin{equation}\label{eq:inclusion}
\mathcal C \subset \bigcup_{i=1}^s \mathcal C_i \subset 
\bigcup_{i=1}^s \mathcal C_i(\varepsilon)\subset \mathcal G,
\end{equation}
where
$\mathcal C_i(\varepsilon)=\{\mu\in\mathcal M\colon \int\vec\phi_i  d\mu>\vec\alpha_i-\vec{\varepsilon}\}$.
Since $F(\mu)\leq -I(\mu)$ for  $\mu\in\mathcal M$,
Proposition~\ref{up} implies 
\[
\limsup_{n\to\infty}\frac{1}{n}\log |\{x\in X\colon\delta_x^n\in\mathcal C_i\}|\leq -\inf_{\mathcal C_i(\varepsilon)}I
+\varepsilon\ \text{ for }1\leq i\leq s.\]
These inequalities and \eqref{eq:inclusion} give
\[
\limsup_{n\to\infty}\frac{1}{n}\log 
|\{x\in X\colon\delta_x^n\in \mathcal C\}|
 \le \max_{1\leq i\leq s}  
\left(-\inf_{\mathcal C_i(\varepsilon)}I\right)+\varepsilon \le
-\inf_{\mathcal G}I+\varepsilon.
\]
Since $\varepsilon>0$ is arbitrary and
$\mathcal G$ is an arbitrary open set containing 
 $\mathcal C$, it follows that
\[
\limsup_{n\to\infty}\frac{1}{n}\log |\{x\in X\colon\delta_x^n\in \mathcal C\}|\leq
\inf_{\mathcal G \supset \mathcal C} (-\inf_{\mathcal G} I)=
-\inf_{\mathcal C}I,\]
as required. The last equality is due to the lower semicontinuity of $I$.
\end{proof}

\subsection*{(Standing hypotheses from $\S$\ref{escape} to $\S$\ref{m-cluster}):}
The rest of this section is entirely dedicated to the proof of Proposition~\ref{up}. We assume
 $f\colon X\to X$ is an $S$-unimodal map with a non-flat critical point $c$, and $\ell\in\mathbb N$, $\vec{\phi}\in C(X)^{\ell}$, 
$\vec{\alpha}\in\mathbb R^{\ell}$ 
in the statement of Proposition~\ref{up} are fixed till the end of $\S$\ref{m-cluster}.

\subsection{Escape estimate in a single renormalization cycle}\label{escape}
Suppose $m(f)\geq1$, and
let $m$ be an integer with $0\leq m\leq m(f)-1$. For $\vec\beta\in\mathbb R^{\ell}$, 
$P\subset K_{m,m+1}$ and $n\in\mathbb N$, define
\[ A_{n}(\vec{\phi},\vec{\beta},P)=
\cA_{n}(\vec{\phi},\vec{\beta})\cap P\cap\bigcap_{k=0}^{n-1}f^{-k}(K_{m,m+1}).\]
This subsection is dedicated to a proof of the next lemma.
\begin{lemma}\label{hyperbolic}
Suppose $m(f)\geq1$
and let $0\leq m\leq m(f)-1$ be an integer.
For any $\varepsilon>0$ there exists $N\geq1$ such that
if $\vec\beta\in\mathbb R^{\ell}$, $P\in\tilde{\mathscr{P}}_m$, $n\geq N$ 
 satisfy $P\subset K_{m,m+1}$ and
$A_{n}(\vec{\phi},\vec{\beta},P)\neq\emptyset$,
 then there exists a measure 
$\mu\in\mathcal M(f)$ such that 
\begin{align}\tag{a}
|A_{n}(\vec{\phi},\vec{\beta},P)|&\leq \exp\left((F(\mu)+\varepsilon)n\right)|P|\quad and\\
\tag{b}
\int\vec\phi d\mu&>\vec\beta-\vec\varepsilon.
\end{align}
\end{lemma}
\begin{proof}
Let $\varepsilon>0$ and let $P\in\tilde{\mathscr{P}}_m$. Since $P\subset K_{m,m+1}$, 
if $p_{m+1}/p_m= 2$ then $f(P)$ contains a point from the orbit of the periodic point in $\partial J_m$. Then we can verify (a) (b) easily,
taking the measure $\mu$ to be the equidistribution on this periodic orbit.
For the rest of the proof of Lemma~\ref{hyperbolic},
we suppose $p_{m+1}/p_m\neq 2$, and
split the proof  into two cases,
either $P\in\mathscr{P}_m$
or 
$P\in \tilde{\mathscr{P}}_m\setminus \mathscr{P}_m$.
\medskip

\noindent{\bf Case 1: } $P\in\mathscr{P}_m$.
We perform an escape estimate relative to $P$.
Let $n\in\mathbb N$ and suppose $A_{n}(\vec{\phi},\vec{\beta},P)\neq\emptyset$. 
Let $\mathscr{P}_{m,n}(P)$ denote the collection of the diffeomorphic pullbacks of elements of $\mathscr{P}_m$ by $f^{n}$ that are contained in $P$
and intersect $\cA_{n}(\vec{\phi},\vec{\beta})$. Note that
every element of $\mathscr{P}_{m,n}(P)$ is contained in $\bigcap_{k=0}^{n-1} f^{-k}(K_{m,m+1})$, and the union of elements of
$\mathscr{P}_{m,n}(P)$ contains $A_{n}(\vec{\phi},\vec{\beta},P)$.

\begin{sublemma}\label{extend}
 For any $\varepsilon>0$
there exists $N\geq1$ such that for every
$Q\in\mathscr{P}_m$,
  every $n\geq N$ and every diffeomorphic pullback 
$W$ of $Q$ by $f^{n}$ 
contained in $P$, there exist 
$k\in\mathbb Z$ with $|k|\leq p_m-1$ and a diffeomorphic pullback $Y$ of $P$ by $f^{n+k+M_m }$ that is contained in $W$ and satisfies \[|Y|\geq \exp\left(-\varepsilon n\right)|W|.\]
\end{sublemma}
\begin{proof}
Since $P$, $Q\in\mathscr{P}_m$, there exist $k_P$, $k_Q\in\{0,\ldots,p_m-1\}$ such that $f^{k_P}(P)$ and $f^{k_Q}(Q)$ are elements of $\mathscr{P}_m$ contained in $L_m$. 
By the definition of $M_m$ in \eqref{Mm-def}, there exists 
a diffeomorphic pullback $W_P$ of $f^{k_P}(P)$ by $f^{M_m }$ that is contained in $f^{k_Q}(Q)$. 
Let $W'_P$ be the diffeomorphic pullback of $W_P$ by $f^{k_Q}$ contained in $Q$. This is a diffeomorphic pullback of $P$ by $f^{M_m+k_Q-k_P}$. Let $Y$ be the pullback of $W'_P$ by $f^n|_{W}$. 
Then $Y$ is a diffeomorphic pullback of $P$ by $f^{n+M_m+k_Q-k_P}$.

By
Proposition~\ref{koebe} with $\varepsilon$ replaced by $\varepsilon/2$, if $n$ is sufficiently large we have 
\[\begin{split}
\frac{|Y|}{|W|}&\geq \exp\left(-\frac{\varepsilon n}{2}\right)\frac{|f^{n}(Y)|}{|f^{n}(W)|}=\exp\left(-\frac{\varepsilon n}{2}\right)\frac{|W_P'|}{|Q|}\\
&\geq \exp\left(-\frac{\varepsilon n}{2}\right)C\frac{|W_{P}|}{|f^{k_Q}(Q)|}\geq \exp\left(-\varepsilon n\right)
\end{split}\]
as required, where $C>0$ is a constant independent of $P$ and $Q$.
 \end{proof}

In view of
Sublemma~\ref{extend},
for each $W\in \mathscr{P}_{m,n}(P)$
we fix an integer $k(W)$ with $|k(W)|\leq p_m-1$ and a  diffeomorphic pullback $Y_W$ of $P$ by $f^{n+k(W)+M_m}$ that is contained in $W$ and satisfies
$|Y_W|\geq \exp\left(-\varepsilon n/3\right)|W|.$ 
Set
\[
\mathscr{P}_{m,n,k}(P)=\{W\in \mathscr{P}_{m,n}(P)\colon k(W)=k\}.\]
If $n$ is sufficiently large, then 
for all $-p_m+1\leq k\leq p_m-1$ we have
\begin{equation}
\label{bir}\bigcup_{W\in \mathscr{P}_{m,n,k}(P)}Y_W\subset A_{n+k+M_m}\left(\vec{\phi},\vec{\beta}-\frac{1}{2}\vec{\varepsilon}\right),
\end{equation}
and
\begin{equation}\label{bir-new}
|A_{n}(\vec{\phi},\vec{\alpha},P)|
\leq\sum_{W\in \mathscr{P}_{m,n}(P)}|W|\leq 
\exp\left(\frac{\varepsilon n}{3}\right)\sum_{W\in \mathscr{P}_{m,n}(P)}|Y_W|.\end{equation}
Choose $\underline{k}\in\{-p_m+1,\ldots,p_m-1\}$ such that 
\begin{equation}\label{bir-new-new}
\sum_{W\in \mathscr{P}_{m,n}(P)}|Y_W|\leq(2p_m-1)\sum_{W\in \mathscr{P}_{m,n,\underline{k}}(P)}|Y_W|,\end{equation}
and set $\tilde n=n+\underline{k}+M_m$. 
The restriction of $f^{\tilde n}$
to $\bigcup_{W\in \mathscr{P}_{m,n,\underline{k}}(P)} Y_W$
induces a fully branched expanding Markov map onto $P$
with finitely many  branches.
Let $\Delta\subset\Gamma_m$ denote its maximal invariant set, namely
\[\Delta=\bigcap_{j=0}^\infty (f^{\tilde n})^{-j}\left(\bigcup_{W\in 
\mathscr{P}_{m,n,\underline{k}}(P)}Y_W\right).\]
The map
$f^{\tilde n}|_{\Delta}\colon \Delta\to \Delta$
is topologically conjugate to the full shift $\sigma\colon\Sigma\to\Sigma$ over the finite alphabet $\mathscr{P}_{m,n,\underline{k}}(P)$. We write  $\pi\colon\Sigma\to\Delta$ for the conjugacy map, and define
the induced potential $\Phi\colon\Sigma\to\mathbb R$ by $\Phi(\omega)=-\log|(f^{\tilde n})'(\pi(\omega))|$.
From Lemma~\ref{u-decay},
$\Phi$ is continuous with respect to the shift metric. 
The  variational principle \cite[p.40, 2.17]{Bow75} gives
\begin{equation}\label{eq:variational_principle}
\sup_{\tilde\nu\in\mathcal M(\sigma)}\left(h(\sigma,\tilde\nu)+
\int\Phi d\tilde\nu\right)=
\lim_{j\to\infty}\frac{1}{j}\log\left(\sum_{\omega\in\sigma^{-j}(\omega')}
\exp{\sum_{i=0}^{j-1}\Phi(\sigma^i\omega)}\right)
\end{equation}
for any fixed $\omega'\in \Sigma$,  where $\mathcal M(\sigma)$ denotes the space of $\sigma$-invariant Borel probability measures endowed with the weak* topology, and $h(\sigma,\tilde\nu)$ denotes the measure-theoretic entropy of $\tilde\nu$ with respect to $\sigma$.
By the distortion estimates in Proposition~\ref{koebe}, if $n\geq1$ is sufficiently large then 
for any $W\in \mathscr{P}_{m,n,\underline{k}}(P)$ and any $\omega\in\Sigma$
 such that $\pi(\omega)\in Y_W$ we have
\[
\exp\left(\Phi(\omega)\right)\geq 
\exp\left(-\frac{\varepsilon n}{2} \right)\frac{|Y_W|}{|P|}.\]
 Hence, the series inside the logarithm in \eqref{eq:variational_principle} is bounded from below as follows:
 \[\begin{split}\sum_{\omega\in \sigma^{-j}(\omega')}\exp\left(
 \sum_{i=0}^{j-1}\Phi(\sigma^i\omega)\right)
&\geq\left(\inf_{\omega'\in \Sigma}\sum_{\omega\in \sigma^{-1}(\omega')}\exp\left(\Phi(\omega)\right)\right)^j\\
&\geq\left(
\exp\left(-\frac{\varepsilon n}{2} \right)\sum_{W\in
\mathscr{P}_{m,n,\underline{k}}(P)}\frac{|Y_W|}{|P|}\right)^j.\end{split}\]
Taking logarithms of both sides, dividing by $j$ and letting $j\to \infty$, we have
\[
\lim_{j\to\infty}\frac{1}{j}\log\left(\sum_{\omega\in \sigma^{-j}(\omega')}
\exp\left(\sum_{i=0}^{j-1}\Phi(\sigma^i\omega)\right)\right)
\geq\log\sum_{W\in \mathscr{P}_{m,n,\underline{k}}(P) }\frac{|Y_W|}{|P|}-\frac{\varepsilon n}{2}.
\]
Plugging this inequality into  \eqref{eq:variational_principle} yields
\begin{equation}\label{plug}\sup_{\tilde\nu\in\mathcal M(\sigma)}\left(h(\sigma,\tilde\nu)+
\int\Phi d\tilde\nu\right)\geq
\log\sum_{W\in\mathscr{P}_{m,n,\underline{k}}(P)  }\frac{|Y_W|}{|P|}-\frac{\varepsilon n}{2}.
\end{equation}
Since $\mathcal M(\sigma)$ is compact in the weak* topology and the entropy function on it is upper semicontinuous, 
there exists a measure $\tilde\mu\in \mathcal M(\sigma)$
that attains the supremum in \eqref{plug}. 
The measure $\tilde\mu\circ\pi^{-1}$ on $\Delta$ is $f^{\tilde n}$-invariant and its spread 
\[
\mu = \frac{1}{\tilde n}
\sum_{W\in  \mathscr{P}_{m,n,\underline{k}}(P)  }\sum_{j=0}^{\tilde n-1}
(\tilde\mu|_{Y_W})\circ (f^j\circ\pi)^{-1}
\]
belongs to $\mathcal M_m(f)$.
Although $f^{\tilde n}$ may not be the first return map to $\Delta$, 
 Abramov's formula, connecting the entropies of $\mu$ and $\tilde\mu\circ\pi^{-1}$, and Kac's formula, connecting the integrals of 
 $-\log|f|$ and 
 $\Phi$, still hold \cite[Theorem~2.3]{PesSen08} and
we have
\begin{equation}\label{abramov}
h(\sigma,\tilde\mu)+\int\Phi d\tilde\mu=F(\mu)\tilde n.
\end{equation}
From \eqref{plug} and \eqref{abramov} we have
\[\label{first}
\sum_{W\in \mathscr{P}_{m,n,\underline{k}}(P) } |Y_W|\leq \exp\left(F(\mu)\tilde n+\frac{\varepsilon n}{2 }\right)|P|.
\]
From this inequality, \eqref{bir-new} and \eqref{bir-new-new}, for all sufficiently large 
 $n\geq1$ we obtain \[\begin{split}
 |A_{n}(\vec{\phi},\vec{\beta},P)|
&\leq 
(2p_m-1)\exp\left(\frac{\varepsilon n}{3} \right)\sum_{W\in \mathscr{P}_{m,n,\underline{k}}(P) }|Y_W|\\
&\leq \exp\left((F(\mu)+\varepsilon)n\right)\left|P\right|,\end{split}
\]
as required in Lemma~\ref{hyperbolic}(a).
The inequality in Lemma~\ref{hyperbolic}(b) follows from \eqref{bir}. 
\medskip

\noindent{\bf Case 2:}  
$P\in \tilde{\mathscr{P}}_m\setminus \mathscr{P}_m$. 
 Let $z$ denote the periodic point of period $p_m$ in $\partial J_m$.
Then $z$ is hyperbolic repelling, and $f^{p_m}$ maps the point in $\partial J_m\setminus\{z\}$ to $z$.
For simplicity we assume $z\in P$.
 Otherwise, $z\in f^{p_m}(P)$ holds and so the argument is analogous.

Let $n\in\mathbb N$ satisfy
$A_n(\vec{\phi},\vec{\beta},P)\neq\emptyset$.
 Our strategy is to use the simplest version of the coarse graining to be formally introduced in $\S$\ref{coarse}:
 we begin by splitting
$A_{n}(\vec{\phi},\vec{\beta},P)$ into
two subsets, one consisting of points which remain in $\bigcup_{k=0}^{p_m-1}f^k(P)$ almost 
until time $n$, and the complement of this set.
For the first set, the influence of the dynamics
near the set $\Gamma_m$ is negligible, and a local analysis near the orbit of $z$ suffices for the estimate of its Lebesugue measure. The second set is influenced by the dynamics near the orbit of $z$ and the dynamics near $\Gamma_m$. We estimate the Lebesgue measure of each set separately, and unify the estimates at the end to obtain the desired one in Lemma~\ref{hyperbolic}.

For each $t\in\mathbb N$, we put
\[V_{t}=P\cap\left(f^{-t}\left(\bigcup_{Q\in\mathscr{P}_m} Q\right)\setminus\bigcup_{k=0}^{t-1} f^{-k}\left(\bigcup_{Q\in\mathscr{P}_m} Q\right)\right),\]
 and split
$|A_{n}(\vec{\phi},\vec{\beta},P)
|=I+I\!I$
where
\[I=\sum_{t=1}^{n }|A_{n}(\vec{\phi},\vec{\beta},V_t)|\ \text{ and }\ I\!I=\sum_{ t= n+1 }^\infty|A_{n}(\vec{\phi},\vec{\beta},V_t)|.\]
The set $V_t$ is non-empty if and only if $t$ is an integer multiple of $p_m$. If $V_t\neq\emptyset$ then $f^t(V_t)$ contains some element of $\mathscr{P}_m$, and thus
\begin{equation}\label{lowcm}
|f^{t}(V_t)|\geq\min\{|Q|\colon Q\in\mathscr{P}_m\}>0.\end{equation}
Since the restriction of $f^t$ to the smallest closed interval that contains $V_t$ and $z$ extends to a diffeomorphism on an open interval,
the distortion of $f^{t}$ on the former interval is bounded by a constant $C>0$ that depends only on $f$. Then we have
\begin{equation}\label{new-1}|V_t| \leq C\exp\left(F(\delta_z^{p_m})t\right)|P|.\end{equation}
 By \eqref{new-1} we obtain
\begin{equation}\label{II-estimate}I\!I\leq\sum_{t=n+1}^\infty|V_t|\leq \frac{C\exp\left(F(\delta_z^{p_m})(n+1)\right)}{1-\exp(F(\delta_z^{p_m})p_m)}|P|.
\end{equation}

In order to estimate the sum $I$, for each 
$t\in\{1,\ldots, n\}$ 
let
 $M_{2,\ell}(t,\mathbb Z_\varepsilon)$ denote the set of $2\times \ell$ matrices $(\beta_{ij})_{i=0,1,1\leq j\leq \ell }$
with entries in 
$\{(\varepsilon/3) a\colon a\in\mathbb Z\}$ for which
the following hold:
\begin{align}\label{v20}
\beta_{0j}t+\beta_{1j}(n-t)&>\left(\beta_j-\frac{\varepsilon}{3}\right)n\ \text{ for }1\leq j\leq\ell;\\
\label{v30}
\inf\phi_j-\frac{\varepsilon}{3}&<\beta_{ij}\leq \sup\phi_j\ \text{ for }i=0,1\text{ and }1\leq j\leq\ell.
\end{align}
For $\mathbf B=(\beta_{ij})\in M_{2,\ell}(t,\mathbb Z_\varepsilon)$ and 
$i=0,1$,
put $\vec{\beta}_i=(\beta_{i1},\ldots,\beta_{i\ell})$ and define
\[A_{n}(t,\mathbf B)=\left\{x\in  V_t\colon \frac{1}{t}S_{t}\vec\phi(x)\geq\vec\beta_0\ \text{ and }\ \frac{1}{n-t}S_{n-t}\vec\phi(f^{t}(x))\geq\vec\beta_1\right\}.\]
We claim that 
\begin{equation}\label{claim-eq}A_{n}(\vec{\phi},\vec\beta,V_t)\subset \bigcup_{\mathbf B\in  M_{2,\ell}(t,\mathbb Z_\varepsilon)}A_n(t,\mathbf B).\end{equation}
Indeed, for each $x\in  A_n(\vec{\phi},\vec{\beta},V_t)$ 
there exists $\mathbf B=(\beta_{ij})\in M_{2,\ell}(t,\mathbb Z_\varepsilon)$
such that for $1\leq j\leq\ell$ we have
\[\beta_{0j}\leq
\frac{1}{t }S_{t}\phi_j(x)<\beta_{0j}+\frac{\varepsilon}{3}\ \text{ and }\  \beta_{1j}\leq \frac{1}{n-t }S_{n-t}\phi_j(f^{t}(x))<\beta_{1j}+\frac{\varepsilon}{3}.
\]
This implies \eqref{v20} and \eqref{v30},  and hence
$x\in A_n(t,\mathbf B)$, which verifies  \eqref{claim-eq}.

Let
 $\delta>0$ be such that
\begin{equation}\label{phi-d}\|\vec\phi(x)-\vec\phi(y)\|<\frac{\varepsilon}{7}\ \text{
for $x,y\in X$ with $|x-y|<\delta$.}\end{equation}
 In view of what was proved in Case~1 and Lemma~\ref{u-decay}, for the rest of the proof in Case~2 we assume $N\geq1$ is a sufficiently large integer for which the following hold: 
\begin{itemize}
\item[(D1)] if $n\geq N$ then the inequalities in
(a) and (b) of Lemma~\ref{hyperbolic} hold for every element of $\mathscr{P}_m$ with $\varepsilon$ replaced by $\varepsilon/4$;
\item[(D2)]for every connected component $W$ of $\bigcap_{k=0}^{N-1}f^{-k}(K_{m,m+1})$ we have $|W|<\delta$.
\end{itemize}
To finish the proof in Case~2, for all sufficiently large $n$
we establish the inequalities in
(a) and (b) of Lemma~\ref{hyperbolic} for every element of $\tilde{\mathscr{P}}_m\setminus \mathscr{P}_m$.

Let $n\geq N$ and let $t\in\{1,\ldots,n-N+1\}$.
   For each matrix $\mathbf B=(\beta_{ij})\in M_{2,\ell}(t,\mathbb Z_\varepsilon)$
with $A_n(t,\mathbf B)\neq \emptyset$,
 we estimate the Lebesgue measure
 $|A_n(t,\mathbf B)|$ from above. 
Let $\mathscr{P}_m(t,\mathbf B)$ denote the collection of the diffeomorphic pullbacks of 
elements of $\mathscr{P}_m$ by $f^{n-t}$
that intersect $f^{t}(A_n(t,\mathbf B))$.
The elements of $\mathscr{P}_m(t,\mathbf B)$ intersect
$\cA_{n-t}(\vec{\phi},\vec\beta_1)$ and their union
 contains $f^{t}(A_n(t,\mathbf B))$,
because
\[f^{t}(A_{n}(t,\mathbf B))\subset \cA_{n-t}(
\vec{\phi},\vec\beta_1)\cap\bigcap_{k=0}^{n-t-1}f^{-k}(K_{m,m+1}).
\]
By (D1), for each $Q\in \mathscr{P}_m(t,\mathbf B)$
 there exists $\mu_{t,Q}\in\mathcal M_m(f)$ such that 
\begin{align}\label{already-10}
|\cA_{n-t}(\vec{\phi},\vec\beta_1)\cap Q
|&\leq\exp\left(\left (F(\mu_{t,Q})+\frac{\varepsilon}{4}\right)(n-t)\right)|Q|\quad\text{and}\\
\label{already1}\int\vec\phi d\mu_{t,Q}&>\vec\beta_1-\frac{1}{4}\vec\varepsilon.
\end{align}
Pick a measure $\mu_{t,\mathbf B}$ in the finite set $\{\mu_{t,Q}\colon Q\in \mathscr{P}_m(t,\mathbf B)\}$ that maximizes the free energy within this finite set. By \eqref{already-10} and $\sum_{Q\in \mathscr{P}_m(t,\mathbf B)}|Q|\leq|X|$ we have
\[
\begin{split}| f^{t}(A_n(t,\mathbf B ))
|&\leq \sum_{Q\in\mathscr{P}_m(t,\mathbf B) } |\cA_{n-t}(\vec{\phi},\vec\beta_1)\cap Q
|\\&\leq\exp\left(\left (F(\mu_{t,\mathbf B})+\frac{\varepsilon}{4}\right)(n-t)\right)|X|.\end{split}\]
Combining this inequality with \eqref{lowcm}, we obtain
\begin{equation}\label{already10}
\begin{split}\frac{|A_n(t,\mathbf B )
|}{|V_t|}&\leq C
\frac{| f^{t}(A_n(t,\mathbf B))
|}{|f^{t}(V_t)|}\leq \frac{C|X|}{|f^{t}(V_t)| }\exp\left(\left (F(\mu_{t,\mathbf B})+\frac{\varepsilon}{4}\right)(n-t)\right).\end{split}\end{equation}
For the first inequality we have used the bounded distortion.

\begin{sublemma}\label{deltapz}If $n$ is sufficiently large, then we have
\[\int\vec\phi d\delta_z^{p_m}>\vec\beta-\vec\varepsilon.\]\end{sublemma}
\begin{proof}
Take $x\in A_{n}(\vec{\phi},\vec{\beta},P)$. Since $x$, $z\in P$, 
the same argument as the deduction of 
\eqref{w-decay-c} yields
\begin{equation}\label{deltapz-eq}\frac{1}{n}S_n\vec\phi(z)\geq \frac{1}{n}S_n\vec\phi(x)-\frac{\vec\varepsilon}{2}>\vec\beta-\frac{\vec\varepsilon}{2}\end{equation}
provided $n$ is sufficiently large.

 Write $n=p_mq+r$, with non-negative integers $q$, $r$ satisfying $0\leq r\leq q-1$.
Using \eqref{deltapz-eq} we have
\[\begin{split}\int\vec\phi d\delta_z^{p_m}&=\frac{1}{p_mq}S_{p_mq}\vec\phi(z)=\frac{n}{p_mq}\frac{1}{n}S_n\vec\phi(z)-\frac{1}{p_mq}S_{n-p_mq}\vec\phi(f^{p_mq}(z))\\
&=\frac{1}{n}S_n\vec\phi(z)+\frac{n-p_mq}{p_mq}\frac{1}{n}S_n\vec\phi(z)-\frac{1}{p_mq}S_{n-p_mq}\vec\phi(f^{p_mq}(z))\\&>\vec\beta-\frac{\vec\varepsilon}{2}-\frac{2r}{p_mq}\|\vec\phi(z)\|>\vec\beta-\vec\varepsilon\end{split}\]
provided $n$ is sufficiently large.
\end{proof}
Define a measure $\nu_{t,\mathbf B}\in\mathcal M(f)$ by
\[\nu_{t,\mathbf B}=
\frac{t}{n}\cdot\delta_z^{p_m}+\left(1-\frac{t}{n}\right)\cdot\mu_{t,\mathbf B}.\]
\begin{sublemma}\label{esc-4}If $n$ is sufficiently large, then we have
\[\int\vec\phi d\nu_{t,\mathbf B}>\vec\beta-\vec\varepsilon.\]
\end{sublemma}
\begin{proof}From
  \eqref{already1} we have \begin{equation}\label{esc-4-1}\int\vec\phi d\mu_{t,\mathbf B}>\vec\beta_1-\frac{1}{4}\vec\varepsilon.\end{equation} 
We first treat the case $t\sup\Vert\vec\phi\Vert<\varepsilon n/30$. Clearly we have
 \begin{equation}\label{esc-4-2}\left\Vert\frac{t}{n}\int\vec\phi d\delta_z^{p_m}\right\Vert<\frac{\varepsilon}{30}.
 \end{equation}
  From \eqref{v30} we have $\|\vec\beta_{0}\|\leq \sup\|\vec\phi\|+\varepsilon/3$, which gives
 \begin{equation}\label{esc-4-10}
 \left\Vert\frac{t}{n}\vec\beta_0\right\Vert<\frac{\varepsilon}{3}+\frac{\varepsilon}{30}.
 \end{equation}
 Combining \eqref{esc-4-1}, \eqref{esc-4-2}, \eqref{esc-4-10} and then using \eqref{v20} we obtain
\[
\begin{split}\int\vec\phi d\nu_{t,\mathbf B}&=\frac{t}{n}\int\vec\phi d\delta_z^{p_m}+\left(1-\frac{t}{n}\right)\int\vec\phi d\mu_{t,\mathbf B}\\
&>-\frac{1}{30}\vec\varepsilon+\frac{t}{n}\vec\beta_0+\frac{n-t}{n}\left(\vec\beta_1-\frac{1}{4}\vec\varepsilon\right)-\frac{t}{n}\vec\beta_0
\\
&\geq-\frac{1}{30}\vec\varepsilon+\left(\frac{t}{n}\vec\beta_0+\frac{n-t}{n}\vec\beta_1\right)-\frac{1}{4}\vec\varepsilon-\frac{t}{n}\vec\beta_0
\\
&>-\frac{1}{30}\vec\varepsilon+\vec\beta-\frac{1}{3}\vec\varepsilon-\frac{1}{4}\vec\varepsilon-\frac{1}{3}\vec\varepsilon-\frac{1}{30}\vec\varepsilon>\vec\beta-\vec\varepsilon,\end{split}
\]
as required.
\medskip

It is left to treat the case $t\sup\Vert\vec\phi\Vert\geq\varepsilon n/30$.
We assume $n$ is large enough so that $t\geq N$.
   Let $x\in A_{n}(t,\mathbf B)$.
   There exists $k\in\{0,\ldots,p_m-1\}$ such that 
   for every $0\leq s\leq t-1$,  $f^s(x)$ and $f^s(f^k(z))$ belong to the same connected component of $K_{m,m+1}$. Hence,
   if $s\leq t-N$ then
    $f^s(x)$ and $f^s(f^k(z))$ belong to the same connected component of $\bigcap_{\ell=0}^{t-s-1}f^{-\ell}(K_{m,m+1})$,
    and so $|f^s(x)-f^s(f^k(z))|<\delta$ by (D2).
From \eqref{phi-d} we have
 \[\begin{split}\Vert  S_{t}\vec\phi(x)-S_{t}\vec\phi(f^k(z))\Vert\leq& \Vert  S_{t-N+1}\vec\phi(x)-S_{t-N+1}\vec\phi(f^k(z))\Vert\\
 &+\Vert  S_{N-1}\vec\phi(f^{t-N+1}(x))-S_{N-1}\vec\phi(f^{t-N+1}(f^k(z)))\Vert\\
 \leq& \frac{\varepsilon }{7}(t-N+1)+2(N-1)\sup\|\vec\phi\|<\frac{t\varepsilon}{6}\end{split}\]
provided $n$ is sufficiently large, and thus
\begin{equation}\label{esc-4-5}\frac{1}{t}S_{t}\vec\phi(f^k(z))>\frac{1}{t} S_{t}\vec\phi(x)-\frac{1}{6}\vec\varepsilon\geq \vec\beta_0-\frac{1}{6}\vec\varepsilon.\end{equation}
Since 
$z$ is of period $p_m$, we have
\[\left\Vert\frac{1}{p_m}S_{p_m}\vec\phi(z)- \frac{1}{t}S_t\vec\phi(f^k(z))\right\Vert\leq \frac{p_m}{t}\sup\|\vec\phi\|<\frac{\varepsilon}{6}\]
provided $n$ is sufficiently large,
and thus
\begin{equation}\label{esc-4-6}\frac{1}{p_m}S_{p_m}\vec\phi(z)> \frac{1}{t}S_t\vec\phi(f^k(z))-\frac{1}{6}\vec\varepsilon.\end{equation}
Combining \eqref{esc-4-5} and \eqref{esc-4-6}
we have
 \begin{equation}\label{esc-4-7}\int\vec\phi d\delta_z^{p_m}=\frac{1}{p_m}S_{p_m}\vec\phi(z)>\vec\beta_0-\frac{1}{3}\vec\varepsilon.\end{equation}
 Combining \eqref{esc-4-1}, \eqref{esc-4-7} and then using \eqref{v20} we obtain
\[\begin{split}\int\vec\phi d\nu_{t,\mathbf B}&=\frac{t}{n}\int\vec\phi d\delta_z^{p_m}+ \frac{n-t}{n}\int\vec\phi d\mu_{t,\mathbf B}\\&>\frac{t}{n}\left(\vec\beta_0-\frac{1}{3}\vec\varepsilon\right)+\frac{n-t}{n}\left(\vec\beta_1-\frac{1}{4}\vec\varepsilon\right)\\
&=\frac{t}{n}\vec\beta_0+\frac{n-t}{n}\vec\beta_1-\frac{1}{3}\vec\varepsilon-\frac{1}{4}\vec\varepsilon>\vec\beta-\vec\varepsilon,\end{split}
\]
as required.
\end{proof}

From \eqref{lowcm}
\eqref{new-1} and \eqref{already10} we have
\begin{equation}\label{vi}
|A_n(t,\mathbf B)|\leq  \frac{C^2|X|}{|f^{t}(V_t)|}\exp\left(\left( F(\nu_{t,\mathbf B})+\frac{\varepsilon}{4}\right)n\right)|P|.
\end{equation}
Let $\mu$ be a measure in the finite set $\{\nu_{t,\mathbf B}\colon\mathbf B\in M_{2,\ell}(t,\mathbb Z_\varepsilon)\}\cup\{\delta_z^{p_m}\}$ that maximizes the free energy within this finite set. By \eqref{v30}, clearly we have 
\begin{equation}\label{count-b}
\#M_{2,\ell}(t,\mathbb Z_\varepsilon)\leq\prod_{j=1}^\ell\left(\frac{3}{\varepsilon}\left(\sup\phi_j-\inf\phi_j\right)+2\right)^2.
\end{equation}
By \eqref{claim-eq},
 if $n$ is sufficiently large then \eqref{vi} 
 and \eqref{count-b} together imply
\begin{equation}\label{viii}|A_{n}(\vec{\phi},\vec\beta,V_t)|\leq\sum_{\mathbf B\in M_{2,\ell}(t,\mathbb Z_\varepsilon)}|A_n(t,\mathbf B)|\leq \exp\left(\left( F(\mu)+\frac{\varepsilon}{3}\right)n\right)|P|.\end{equation}

To finish, from \eqref{new-1} and \eqref{viii} we have
\[\begin{split}I&=\sum_{t=1}^{n-N}|A_{n}(\vec{\phi},\vec{\beta},V_t)|+
\sum_{t= n-N+1 }^{n}|A_{n}(\vec{\phi},\vec{\beta},V_t)|\\
&\leq\sum_{t=1}^{n-N}|A_{n}(\vec{\phi},\vec{\beta},V_t)|+
\sum_{t= n-N+1}^n |V_t|\\
&\leq \frac{n}{p_m}\exp\left(\left( F(\mu)+\frac{\varepsilon}{3}\right)n\right)|P|+\frac{Cn}{p_m} \exp\left(F(\mu)(n-N+1)\right)|P|.\end{split}\]
Combining this estimate with that of the sum $I\!I$ in \eqref{II-estimate}, we obtain
\[\begin{split}|A_{n}(\vec{\phi},\vec{\beta},P)
|=I+I\!I\leq 
\exp\left((F(\mu)+\varepsilon)n\right)|P|
\end{split}\]
for all sufficiently large $n$
as required in Lemma~\ref{hyperbolic}(a).
 The inequality in Lemma~\ref{hyperbolic}(b) is a consequence of Sublemmas~\ref{deltapz} and \ref{esc-4}.
 The proof of Lemma~\ref{hyperbolic} is complete.
\end{proof}

\subsection{Coarse graining decomposition into clusters}\label{coarse}
In this subsection, for all sufficiently large $n\geq1$
we decompose the set $\cA_n(\vec{\phi},\vec{\alpha})$ into a finite number of clusters consisting of points that share the same `itinerary' up to time $n$, and share almost the same average values of $\vec{\phi}$ along the segments of orbits each contained in a single cycle. 

Suppose $m(f)\geq1$. Let $\varepsilon>0$ and
write $\mathbb Z_\varepsilon=\{(\varepsilon/3) a\colon a\in\mathbb Z\}$.
If $m(f)=\infty$, then let $m_*$, $N_*\in\mathbb N$ satisfy the conclusion of Lemma~\ref{deep} with
 $\varepsilon$ replaced by $\varepsilon/2$, for any $\vec{\beta}=(\beta_{1},\ldots,\beta_{\ell})\in \mathbb Z_\varepsilon^\ell$ satisfying
$\inf\phi_j-\varepsilon/3<\beta_{j}\leq \sup\phi_j$ for $1\leq j\leq\ell.$
We set
\begin{equation}\label{M-def}M=\begin{cases}m_*&\ \text{ if $m(f)=\infty$,}\\
m(f)&\ \text{ if $1\leq m(f)<\infty$.}\end{cases}\end{equation}

For each $x\in X$, we define by induction two (possibly infinite) increasing sequences $\{\hat n_i(x)\}_{i=0}^{\hat q(x)}$,
$\{\hat m_i(x)\}_{i=0}^{\hat q(x)}$ of non-negative integers
that successively record the time and position at which the orbit of $x$ falls into deeper renormalization cycles.
Start with 
\[\hat n_0(x)=0\quad\text{and}\quad 
\hat m_0(x)=
\max\{0\leq m\leq M\colon x\in K_{m}\}. 
\]
Let $i\geq0$ and
suppose $\hat n_i(x)$, $\hat m_i(x)$ are defined. If $\hat m_{i}(x)=M$, or if $\hat m_i(x)<M$ and 
$f^k(x)\in K_{\hat m_i(x),\hat m_i(x)+1}$
for every $k>\hat n_i(x)$, then we stop the inductive definition by setting $\hat q(x)=i$. Otherwise, we define
\[\begin{split}
\hat n_{i+1}(x)&=\min\{k> \hat n_i(x)\colon f^k(x)\notin
K_{\hat m_i(x),\hat m_i(x)+1}\}\quad \text{and}\\
\hat m_{i+1}(x)&=\max\{\hat m_i(x)+1\leq m\leq M\colon f^{\hat n_{i+1}(x)}(x)\in K_{m}\}.\end{split}\]
If $\hat m_i(x)$, $\hat n_i(x)$ are defined for all $i\geq1$, then we set $\hat q(x)=\infty$.

For integers $n>M$ and 
 $0\leq q\leq M$, we define
\[I_n(q)=
\left\{
\begin{tabular}{l}
\!\!\!$\mathbf{t}=((n_0,m_0), (n_1,m_1),\ldots,(n_q,m_q))\in(\mathbb Z^2)^{q+1}\colon $\!\!\!\\
$\quad\quad0=n_0< n_1<\cdots<n_q< n$,\\ $\quad\quad0\le m_0<m_1<\cdots<m_q\leq M$\end{tabular}
\right\}.
\]
Let $\mathbf t=((n_0,m_0),\ldots,(n_q,m_q))\in I_n(q).$ For convenience we set 
\[n_{q+1}=n,\] 
 and call
 $n_0,n_1,\ldots,n_{q+1}$ {\it transition times}.
 Further we set \begin{equation}\label{ti-eq}t_i=n_{i+1}-n_i\ \text{ for  }0\leq i\leq q.\end{equation}
 Define
\[R(\mathbf t)=
\left\{
\begin{tabular}{l}
\!\!\!$x\in X\colon $\!\!\! $\hat q(x)\ge q$, $(\hat n_{i}(x),\hat m_i(x))=(n_i,m_i)$ for  $0\leq i\leq q$\!\!\!\\
\quad\quad\quad and $\hat n_{q+1}(x)\ge n$ if $\hat q(x)>q$ \end{tabular}
\right\}.
\]
We have 
\begin{equation}\label{X-decomposition}
X=
\bigcup_{q=0}^{M}\bigcup_{\mathbf{t}\in I_n
(q)}R(\mathbf{t}).
\end{equation}

We decompose the set $\cA_{n}(\vec{\phi},\vec{\alpha})\cap R(\mathbf{t})$ with respect to coarse grained values at scale $\varepsilon$ of time averages of $\vec\phi$  between two consecutive transition times.
For each $\mathbf{t}\in I_n(q)$,
let 
$M_{q+1,\ell}(\mathbf t,\mathbb Z_\varepsilon)$ 
 denote the set of $(q+1)\times\ell$ matrices $\mathbf A=(\alpha_{ij})_{0\leq i\leq q, 1\leq j\leq\ell}$
with entries in $\mathbb Z_\varepsilon$ such that the following hold:
\begin{align}\label{v2}
\sum_{i=0}^q\alpha_{ij}t_i&>\left(\alpha_j-\frac{\varepsilon}{3}\right)n\ \text{ for  }1\leq j\leq \ell;\\
\label{v3}
\inf\phi_j-\frac{\varepsilon}{3}&<\alpha_{ij}\leq \sup\phi_j\ \text{ for }
0\leq i\leq q\ \text{ and }\ 1\leq j\leq\ell.
\end{align}
For a matrix $\mathbf A=(\alpha_{ij})\in M_{q+1,\ell}(\mathbf t,\mathbb Z_\varepsilon)$,
let $\vec{\alpha}_i=(\alpha_{i1},\ldots,\alpha_{i\ell})$ denote the $i$-th row of $\mathbf A$.
Define
\[R(\mathbf{t},\mathbf A)=
\{x\in R(\mathbf{t})\colon S_{t_i }
\vec\phi(f^{n_i}(x))
\geq t_i\vec{\alpha}_{i}\ \text{ for  }0\leq i\leq q\}.\]
 We claim that 
 \begin{equation}\label{claimeq2}
\cA_n(\vec{\phi},\vec{\alpha})\cap R(\mathbf{t})\subset 
\bigcup_{\mathbf A\in M_{q+1,\ell}(\mathbf t,\mathbb Z_\varepsilon)} R(\mathbf{t},\mathbf A).
\end{equation}
Indeed, for each $x\in \cA_n(\vec{\phi},\vec{\alpha})\cap R(\mathbf{t})$ 
there exists
$\mathbf A\in M_{q+1,\ell}(\mathbf t,\mathbb Z_\varepsilon)$ 
such that for every entry $\alpha_{ij}$ of $\mathbf A$
we have
\[
\alpha_{ij}\leq \frac{1}{t_i }S_{t_i}\phi_j(f^{n_i}(x))<\alpha_{ij}+\frac{\varepsilon}{3}
.\]
This implies \eqref{v2} and \eqref{v3}, which verifies \eqref{claimeq2}.

\subsection{Lebesgue measures of clusters}\label{m-cluster}
We are in position to state and prove a main technical estimate on the Lebesgue measure of each cluster $R(\mathbf{t},\mathbf A)$.

\begin{prop}\label{over-lem}
Suppose $m(f)\geq1$.
For any $\varepsilon>0$ there exists an integer $N>M$ such that if $n\geq N$ and $0\leq q\leq M$, then for every
 $\mathbf{t}\in I_n(q)$ and
every $\mathbf A\in M_{q+1,\ell}(\mathbf t,\mathbb Z_\varepsilon)$ satisfying
  $R(\mathbf{t},\mathbf A)\neq\emptyset$, 
there exists a measure $\mu\in\mathcal M(f)$ satisfying
\begin{align}\tag{a}
\left|R(\mathbf{t},
\mathbf A)\right|&\leq 
\exp\left(\left(F(\mu)+\varepsilon\right) n\right)|X|\quad\text{and}\\
\tag{b}
\int\vec\phi d\mu&>\vec{\alpha}-\vec{\varepsilon}.
\end{align}
\end{prop} 
\begin{proof}
Let $\varepsilon>0$, and
let $\delta>0$ be such that
\begin{equation}\label{initial-phi}\|\vec\phi(x)-\vec\phi(y)\|<\frac{\varepsilon}{3}\ \text{
for $x,y\in X$ with $|x-y|<\delta$.}\end{equation}
We first point out that the following three exceptional cases with no transition between renormalization cycles can be treated by much simpler versions of later arguments.\medskip

\noindent{\bf Case~E1:} {\it $q=0$, $m_0=M$, $m(f)<\infty$ and ${\rm int}(J_{m(f)})$ contains a two-sided attracting  periodic point 
(Remark~\ref{rem:C}(C-I)).} This case is treated by a much simpler version of the argument in
Case~2 in Step~2 below.
\medskip

\noindent{\bf Case~E2:} {\it $q=0$, $m_0=M$, $m(f)<\infty$ and $\partial J_{m(f)}$ contains a one-sided attracting  periodic point 
(Remark~\ref{rem:C}(C-II)).}
This case is treated by a much simpler version of the argument in
Case~3 in Step~2 below.
\medskip

\noindent{\bf Case~E3:} {\it $q=0$, $m_0=M$, $m(f)<\infty$ and $f$ has no attracting periodic point
((A) in \S\ref{metric}).}
This case is treated by a much simpler version of the argument in
Case~4 in Step~2 below.
 \medskip

\noindent Hence, we refer the reader to Step~2 below, for details regarding proofs in cases E1, E2, E3. 

We now move on to all the remaining cases with transitions between renormalization cycles.
We are concerned with sufficiently large integers $N_0$, $N_1$ the purposes of which are as follows:
\begin{itemize}
\item $N_0$ concerns the number of iterations in one renormalization cycle, chosen depending only on $\delta$ in \eqref{initial-phi}. In view of Lemma~\ref{u-decay},
 we take $N_0$ such that for every integer
 $0\leq m\leq M-1$ and
   every connected component $W$ of $\bigcap_{k=0}^{N_0-1}f^{-k}(K_{m,m+1})$, $|W|<\delta$ holds;

\item $N_1$ is a lower bound on consecutive transition times such  that some estimates go through (see Lemmas~\ref{inclusion} and \ref{lem-step}), which is chosen depending only on $\vec\phi$, $\varepsilon$, $\delta$, $N_0$, $f$ and satisfies
\begin{equation}\label{N1N0}N_1\geq N_0.\end{equation}

\end{itemize}

The rest of the proof of Proposition~\ref{over-lem} breaks into three steps.
\medskip

\noindent{\bf Step 1: Estimate of error bound in transition.}
Let $n>M$, $0\leq q\leq M$ and let $\mathbf{t}\in
I_n(q)$, 
 $\mathbf A\in M_{q+1,\ell}(\mathbf t,\mathbb Z_\varepsilon)$ satisfy
$R(\mathbf{t},\mathbf A)\neq\emptyset$. 
Since we have excluded cases E1, E2, E3, we have $m_0+1\leq M$ and so $K_{m_0,m_0+1}$ is defined.
Let $\mathscr{R}_0(\mathbf A)$ denote the collection of the connected components of 
$K_{m_0,m_0+1}$. Let $R_0(\mathbf A)$ denote the union of elements of $\mathscr{R}_0(\mathbf A)$.
We have $R_0(\mathbf A)=K_{m_0,m_0+1}$.
For $i=1,\ldots,q+1$ we define a collection $\mathscr{R}_i(\mathbf A)$ of non-degenerate closed subintervals of $R_0(\mathbf A)$ as follows: 
\begin{itemize}
\item  
$\mathscr{R}_i(\mathbf A)$ for $1\leq i\leq q$
is the collection of the
diffeomorphic pullbacks of $f^k(J_{m_i})$ for some 
$k\in\{0,\ldots, p_{m_i}-1\}$
by $f^{n_{i}}$ that are contained in $R_0(\mathbf A)$
and intersect $R(\mathbf{t},\mathbf A)$;
\item if $m_q< M$, then $\mathscr{R}_{q+1}(\mathbf A)$
is the collection of the 
diffeomorphic pullbacks of 
 $f^k(J_{M})$ for some 
$k\in\{0,\ldots, p_{M}-1\}$
by $f^n$ that are contained in $R_0(\mathbf A)$
and intersect $R(\mathbf{t},\mathbf A)$;
\item if $m_q=M$ and $m(f)=\infty$, or if $m_q=M$, $m(f)<\infty$ and $f$ has an attracting periodic point,
then let $\mathscr{R}_{q+1}(\mathbf A)$ denote the collection of the pullbacks of $f^k(J_{M})$ for some 
$k\in\{0,\ldots, p_{M}-1\}$
by $f^n$ that are contained in $R_0(\mathbf A)$
and intersect $R(\mathbf{t},\mathbf A)$.
If $m_q=M$, $m(f)<\infty$ and $f$ has no attracting periodic point, then 
 let $\mathscr{R}_{q+1}'(\mathbf A)$ denote the collection of the pullbacks of $f^k(L_{M})$ for some 
$k\in\{0,\ldots, p_{M}-1\}$
by $f^n$ that are contained in $R_0(\mathbf A)$
and intersect $R(\mathbf{t},\mathbf A)$. Then
 $\mathscr{R}_{q+1}(\mathbf A)$ is the refinement of 
$\mathscr{R}_{q+1}'(\mathbf A)$ by points in $\bigcup_{k=n_q}^{n-1} f^{-k}(c)$.

\end{itemize}
\begin{remark}\label{dif-remark}
Recall that $n>n_q$, and see \eqref{KMF} to motivate the definition of $\mathscr{R}_{q+1}(\mathbf A)$ in the case $m_q=M$.
In this case, 
for $R\in\mathscr{R}_{q+1}(\mathbf A)$, $f^n|_{R}$ is may not be a diffeomorphism.
However,
 for every $R\in\mathscr{R}_{q+1}(\mathbf A)$,
 $f^{n_q}|_R$ is a diffeomorphism. The latter will suffice for bounded distortion arguments to be developed later in the proof of Proposition~\ref{over-lem}:
see Case~2 and Case~4 in Step~2.\end{remark}
For $1\leq i\leq q+1$ let $R_i(\mathbf A)$ denote the union of elements of $\mathscr{R}_i(\mathbf A)$.
Note that
\begin{equation}\label{nest-seq}
R(\mathbf{t},\mathbf A)\subset R_{q+1}(\mathbf A)\subset
\cdots\subset R_1(\mathbf A)\subset R_0(\mathbf A). 
\end{equation}

Among all transition times,
for our purpose it suffices to concentrate on
consecutive ones that are separated by $N_1$.

\begin{lemma}\label{inclusion}
If
 $N_1$ is sufficiently large then the following statements hold:
\begin{itemize}
\item[(a)]  
for every $0\leq i\leq q$ with $t_i>N_1$ and $m_i<M$, 
\[
f^{n_i}(R_{i+1}(\mathbf A))\subset \cA_{t_i }\left(\vec{\phi},\vec{\alpha}_{i} -\frac{1}{2}\vec{\varepsilon}\right)\cap K_{m_i};
\]
\item[(b)] if $t_q>N_1$, $m_q=M$ and \begin{itemize}
  \item[(i)]  $m(f)=\infty$, or 
  \item[(ii)] $m(f)<\infty$ and  $J_{m(f)}$ contains a one-sided attracting periodic point, or 
  \item[(iii)]  $f$ has no attracting periodic point,
    \end{itemize}
 then 
\[
f^{n_q}(R_{q+1}(\mathbf A))\subset \cA_{t_q }\left(\vec{\phi},\vec{\alpha}_{q} -\frac{1}{2}\vec{\varepsilon}\right)\cap K_{m_q}.
\]\end{itemize}
\end{lemma}   
\begin{proof} Let $0\leq i\leq q$ and suppose $t_i>N_1$.
Since the definition of $\mathscr{R}_{i+1}(\mathbf A)$ yields $f^{n_i}(R_{i+1}(\mathbf A))\subset  K_{m_i}$, 
it suffices to show
\begin{equation}\label{goal-eq}f^{n_i}(R_{i+1}(\mathbf A))\subset \cA_{t_i }\left(\vec{\phi},\vec{\alpha}_{i} -\frac{1}{2}\vec{\varepsilon}\right).\end{equation}
To this end, for each $R\in\mathscr{R}_{i+1}(
\mathbf A)$ we fix $x_R\in f^{n_i}(R\cap R(\mathbf t,\mathbf A))$.
The definition of $\mathscr{R}_{i+1}(\mathbf A)$ and that of $R(\mathbf t,\mathbf A)$ give \begin{equation}\label{inclusion-eq1}S_{t_i }\vec\phi(x_R)\geq t_i\vec{\alpha}_{i}.\end{equation}

Suppose $m_i<M$.
By \eqref{N1N0}, $t_i>N_1$ and \eqref{ti-eq}, we have 
\[0\leq n_{i}+N_1-N_0\leq n_{i}+t_i-N_0-1=n_{i+1}-N_0-1,\]
and so for every $R\in\mathscr{R}_{i+1}(
\mathbf A)$, 
\[\bigcup_{k=0}^{t_i-N_0-1}f^{n_i+k}(R)\subset \bigcap_{k=0}^{N_0-1} f^{-k}(K_{m_i,m_{i}+1}).\]
Moreover, the choice of $N_0$ yields $|f^{n_i+k}(R)|<\delta$ for $0\leq k\leq t_i-N_0-1$.
For any $x\in f^{n_i}(R)$,
  \eqref{initial-phi} gives
\begin{equation}\label{inclusion-eq2}\|S_{t_i -N_0}\vec\phi(x)- S_{t_i -N_0}\vec\phi(x_R)\|\leq
\frac{\varepsilon}{3}(t_i -N_0).\end{equation}
For the remaining $N_0$ iterations, clearly we have 
\begin{equation}\label{inclusion-eq3}
\|S_{N_0}\vec\phi(f^{t_i-N_0}(x))- S_{N_0}\vec\phi(f^{t_i-N_0}(x_R))\|
\leq 2N_0 \sup\|\vec\phi\|.
\end{equation}
Combining \eqref{inclusion-eq1}, \eqref{inclusion-eq2}, \eqref{inclusion-eq3} and dividing the result by $t_i$, we obtain
\[
\frac{1}{t_i }S_{t_i  }\vec\phi (x)\geq \vec{\alpha}_{i}-\frac{1}{2}\vec{\varepsilon}
\]
provided $N_1$ is sufficiently large depending only on $\vec\phi$, $\varepsilon$, $N_0$.
Since $R\in\mathscr{R}_{i+1}(\mathbf A)$ and
 $x\in f^{n_i}(R)$
are arbitrary, \eqref{goal-eq} follows. We have verified Lemma~\ref{inclusion}(a). \medskip

To prove Lemma~\ref{inclusion}(b), suppose $t_q>N_1$ and  $m_q=M$. 
 In case (i), the same argument as in the proof of Lemma~\ref{inclusion}(a) remains valid to show \eqref{goal-eq} with $i=q$. In case (ii),
Since $n-n_q=t_q >N_1$, 
if $N_1$ is sufficiently large then 
the uniform convergence in \eqref{unic1} and equation \eqref{initial-phi} together
imply \begin{equation}\label{case2-neweq}\sup_{x\in f^{n_q}(R)}\|S_{t_q}\vec\phi(x)- S_{t_q}\vec\phi(x_R)\|\leq
\frac{\varepsilon} {2}t_q\ 
\text{ for every }R\in\mathscr{R}_{q+1}(
\mathbf A).\end{equation}  
From 
  \eqref{inclusion-eq1} and \eqref{case2-neweq} we obtain \eqref{goal-eq} with $i=q$. 

In case (iii), $M=m(f)$ and $f^{p_M}|_{L_M}$ is topologically exact. 
Recall that we have chosen $\delta>0$ so that \eqref{initial-phi} holds.
There is $n(\delta)\in\mathbb N$ such that for any integer $0\leq j\leq p_M-1$ and any interval $J$ such that  $J\subset f^j(L_M)$ and $|J|\geq \delta$, there is $r\in\{0,\ldots,p_M-1\}$ such that $f^{p_Mn(\delta)+r}(J)=L_M$. We choose $N_1$ so that
\[N_1\geq p_M(n(\delta)+1).\]
Then we have $n-n_q-p_M(n(\delta)+1)>t_q-N_1>0$.
We claim that if
$R\in\mathscr{R}_{q+1}(
\mathbf A)$ and $k\in\{0,\ldots,n-n_q-p_M(n(\delta)+1)\}$
then $|f^{k}(f^{n_q}(R))|<\delta$, 
 for otherwise
$f^{p_Mn(\delta)+r+k}(f^{n_q}(R))=L_M$ would hold for some $r\in\{0,\ldots,p_M-1\}$ and  $p_Mn(\delta)+r+k\leq n-n_q-1$, and as a result $f^n|_R$ would not be injective, a contradiction.

The rest of the argument is similar to the proof of Lemma~\ref{inclusion}(a). By the above claim and \eqref{initial-phi}, 
for every $R\in\mathscr{R}_{q+1}(
\mathbf A)$ and any $x\in f^{n_q}(R)$ we have
\begin{equation}\label{inclusion-eq2000}\|S_{n-n_q-p_M(n(\delta)+1)}\vec\phi(x)- S_{n-n_q-p_M(n(\delta)+1) }\vec\phi(x_R)\|\leq
\frac{\varepsilon}{3}(n-n_q-p_M(n(\delta)+1))).\end{equation}
For the remaining $p_M(n(\delta)+1)$ iterations, clearly we have 
\begin{equation}\label{inclusion-eq3000}
\begin{split}&\|S_{p_M(n(\delta)+1)}\vec\phi(f^{n-n_q-p_M(n(\delta)+1)}(x))- S_{p_M(n(\delta)+1)}\vec\phi(f^{n-n_q-p_M(n(\delta)+1)}(x_R))\|\\
&\leq 2p_M(n(\delta)+1) \sup\|\vec\phi\|.\end{split}
\end{equation}
Combining \eqref{inclusion-eq1}, \eqref{inclusion-eq2000}, \eqref{inclusion-eq3000} and dividing the result by $t_q$ we obtain
\[
\frac{1}{t_q }S_{t_q  }\vec\phi (f^{n_q}(x))\geq \vec{\alpha}_{q}-\frac{1}{2}\vec{\varepsilon}
\]
provided $N_1$ is sufficiently large.
Since $R\in\mathscr{R}_{q+1}(\mathbf A)$ and
 $x\in f^{n_q}(R)$
are arbitrary, \eqref{goal-eq} with $i=q$ follows.
The proof of Lemma~\ref{inclusion} is complete.
\end{proof}

\noindent{\bf Step 2: Estimate of $i$-step conditional probability.}
The next lemma provides estimates of conditional probabilities at consecutive transition times
separated by $N_1$.

\begin{lemma}\label{lem-step}
If $N_1$ is sufficiently large, then
for every $0\leq i\leq q$ with $t_i>N_1$, there exists a measure $\mu_i\in\mathcal M(f)$ such that
\begin{align}\tag{a}
\frac{|R_{i+1}(\mathbf A)|}{|R_i(\mathbf A)|}
&\leq  \exp\left(\left(F(\mu_i)+\varepsilon\right)t_i\right)\quad\text{and}\\
\tag{b}
\int\vec\phi d\mu_{i}&>
\vec{\alpha}_{i}- \vec{\varepsilon}.
\end{align}
\end{lemma}

\begin{proof}
Let $0\leq i\leq q$ and suppose $t_i>N_1$.
In view of the definition of $\mathscr{R}_i(\mathbf A)$ in Step~1, 
we treat the case $m_i<M$ and the case $m_i=M$ (hence $i=q$) separately.

In the former case,
let $R\in\mathscr{R}_i(\mathbf A)$ and suppose $R\cap R_{i+1}(\mathbf A)\neq\emptyset$.  Let $\mathscr{P}_{m_i}(\mathbf A,R)$ denote the collection of the
diffeomorphic pullbacks of $f^k(J_{m_{i+1}})$
for some $k\in\{0,\ldots, p_{m_{i+1}}-1\}$
by $f^{t_i}$
that are contained in 
$f^{n_i}(R\cap R_{i+1}(\mathbf A) )$. 
The union of elements of 
$\mathscr{P}_{m_i}(\mathbf A,R)$
equals $f^{n_i}(R\cap R_{i+1}(\mathbf A) )$.
By Lemma~\ref{inclusion}(a), every element of $\mathscr{P}_{m_i}(\mathbf A,R)$ is contained in
$\cA_{t_i }(\vec{\phi},\vec{\alpha}_{i} -(1/2)\vec{\varepsilon})\cap K_{m_i}$ provided $N_1$ is sufficiently large.
Moreover, for every $Q\in\mathscr{P}_{m_i}(\mathbf A,R)$ there exists $P\in\tilde{\mathscr{P}}_{m_i}$ such that $P\subset K_{m_i,m_i+1}$ and $Q\subset P\subset f^{n_i}(R)$.
Let
$\tilde{\mathscr{P}}_{m_i}(\mathbf A,R)$
denote the collection of elements of $\tilde{\mathscr{P}}_{m_i}$ that 
is contained in $f^{n_i}(R)$ and contains an element of $\mathscr{P}_{m_i}(\mathbf A,R)$.
For each $P\in\tilde{\mathscr{P}}_{m_i}(\mathbf A,R)$ we have
\[\bigcup_{\substack{Q\in \mathscr{P}_{m_i}(\mathbf A,R) \\ Q\subset P}} Q\subset A_{t_i }\left(\vec\phi,\vec{\alpha}_{i}
     -\frac{1}{2}\vec{\varepsilon},P\right).\]
    By Lemma~\ref{hyperbolic}, there exists $\nu_{R,P}\in\mathcal M(f)$ such that
\begin{equation}\label{last-eq400}
 \sum_{\substack{Q\in \mathscr{P}_{m_i}(\mathbf A,R) \\ Q\subset P}}|Q|\leq\left|A_{t_i }\left(\vec\phi,\vec{\alpha}_{i}
     -\frac{1}{2}\vec{\varepsilon},P\right)\right|
     \leq \exp\left(\left(F(\nu_{R,P})+\frac{\varepsilon}{2}\right)t_i\right)|P|\end{equation}
     and
\begin{equation}\label{last-eq100}
\int\vec\phi d\nu_{R,P}>
\vec{\alpha}_{i}- \vec{\varepsilon}.
\end{equation}
Let $\mu_i\in \mathcal M(f)$ be a measure 
in the finite set \[\{\nu_{R,P}\colon R\in\mathscr{R}_i(\mathbf A),\ R\cap R_{i+1}(\mathbf A)\neq\emptyset,\ P\in \tilde{\mathscr{P}}_{m_i}(\mathbf A,R) \}\] that maximizes the free energy within this finite set. By \eqref{last-eq400} we have
\begin{equation}\label{last-eq1000}
\begin{split}|f^{n_i}
(R\cap R_{i+1}(\mathbf A))|&=\sum_{P\in\mathscr{P}_{m_i}(\mathbf A,R) }|P|\\
&\leq
\sum_{P\in\tilde{\mathscr{P}}_{m_i}(\mathbf A,R) }\left|A_{t_i }\left(\vec\phi,\vec{\alpha}_{i}
     -\frac{1}{2}\vec{\varepsilon},P\right)\right|\\
     &\leq\sum_{P\in \tilde{\mathscr{P}}_{m_i}(\mathbf A,R) }\exp\left(\left(F(\nu_{R,P})+\frac{\varepsilon}{2}\right)t_i\right)|P|\\
&\leq\exp\left(\left(F(\mu_i)+\frac{\varepsilon}{2}\right)t_i\right)|f^{n_i}(R)|.\end{split}
\end{equation}
If $i\geq1$, then we have
$\bigcup_{k=n_j}^{n_{j+1}-1}f^k(R)\subset K_{n_j,n_j+1}$ for $0\leq j\leq i-1$.
By Proposition~\ref{koebe}(a), the distortion of the composition $f^{n_i}=f^{t_{i-1}}\circ\cdots \circ f^{t_1}\circ f^{t_0}$ on $R$ is bounded by the constant
\begin{equation}\label{gamma-i}
D_i=\prod_{j=0}^{i-1}\gamma_{m_j}.
\end{equation}
Hence, from \eqref{last-eq1000} we obtain
\[\begin{split}
\frac{|R\cap R_{i+1}(\mathbf A)|}{|R|}\leq D_i\frac{|f^{n_i}
(R\cap R_{i+1}(\mathbf A))|}
     {|f^{n_i}(R)|}&\leq  \exp\left(\left(F(\mu_i)+\varepsilon\right)t_i\right)
\end{split}\]   
provided $N_1$ is sufficiently large. Since $n_0=0$, the same upper bound remains valid for $i=0$. Since $R$ is an arbitrary element of $\mathscr{R}_i(\mathbf A)$ intersecting $R_{i+1}(\mathbf A)$,  
we obtain
\[
\frac{|R_{i+1}(\mathbf A)|}{|R_i(\mathbf A)|}
\leq\max_{\substack{R\in \mathscr{R}_i(\mathbf A)\\ R\cap R_{i+1}(\mathbf A)\neq\emptyset}}\frac{|R\cap R_{i+1}(\mathbf A)|}{|R|}\leq  \exp\left((F(\mu_i)+\varepsilon)t_i \right),
\]
as required in Lemma~\ref{lem-step}(a). 
The inequality in Lemma~\ref{lem-step}(b) is a consequence of 
\eqref{last-eq100}. 
\medskip

To complete the proof of Lemma~\ref{lem-step}, it is left to treat the case
$m_i=M$. In this case we have $i=q$, and there are four subcases.
\medskip

\noindent{\bf Case~1:} $m(f)=\infty$. 
We have $M=m_*$ by \eqref{M-def}.
Set $\mu_q=\mu_\infty$, which is the post-critical measure. 
The inequality in Lemma~\ref{lem-step}(a) with $i=q$ follows from $R_{q+1}(\mathbf A)\subset R_q(\mathbf A)$ and $F(\mu_q)=0$.
 Since $t_q>N_1$ and
 $f^{n_q}(R_{q+1}(\mathbf A))$ is contained in $\cA_{t_q } (\vec\phi,\vec\alpha_q-(1/2)\vec\varepsilon)\cap K_M$
 by Lemma~\ref{inclusion}(b-i),
 the inequality in Lemma~\ref{lem-step}(b) with $i=q$ follows from
Lemma~\ref{deep}(a) provided $N_1\geq N_*$ and $N_1$ is sufficiently large.
\medskip

\noindent{\bf Case~2:} $m(f)=M$ 
 and
${\rm int}(J_{m(f)})$ contains a two-sided attracting  periodic point (Remark~\ref{rem:C}(C-I)). We have
$K_{M}={\rm cl}(B(f))$
by \eqref{KMF} and Lemma~\ref{stop}.
Fix a closed subinterval $J$ of $J_M$ such that 
$c,z_f\in {\rm int}(J)$, $f^{p_{M}}(J)\subset J$ and $|K_M\setminus K_{M+1}|<\delta$ where $K_{M+1}=\bigcup_{k=0}^{p_{M}-1}f^k(J).$

Recall that $\nu$ denotes the element of $\mathcal M_{M-1}(f)$  supported on the orbit of the hyperbolic repelling periodic point in $\partial J_{M}$.
The rest of the proof is similar in spirit to   the argument in Case~2 in the proof of Lemma~\ref{hyperbolic}.  For each $R\in\mathscr{R}_{q}(\mathbf A)$ and each integer $t$ with $0\leq t\leq t_q$, define
\[V_{R,t}=\begin{cases}f^{n_q}(R)\cap K_{M+1}&\text{if }t=0,\\f^{n_q}(R)\cap\left(f^{-t}(K_{M+1})\setminus\bigcup_{k=0}^{t-1} f^{-k}(K_{M+1})\right)&\text{if }1\leq t\leq t_q-1,\\f^{n_q}(R)\setminus \bigcup_{k=0}^{t_q-1} f^{-k}(K_{M+1})&\text{if }t=t_q.\end{cases}\]Notice that $f^{n_q}(R)= \bigcup_{t=0}^{t_q} V_{R,t}$.
Then $V_{R,t}$ has at most two connected components, which are non-degenerate intervals. 
Let $W_{R,t}$ denote union of the diffeomorphic pullbacks of the connected components of $V_{R,t}$ by $f^{n_q}$ that are contained in $R$. 
Define $\nu_{R,t}\in\mathcal M(f)$ by \[\nu_{R,t}=\frac{t}{t_q }\cdot\nu+\left(1-\frac{t }{t_q }\right)\cdot\delta_{O(f)}.\]
  There exists a constant $C\geq1$ depending only $f$, $\delta$ such that
  if $x\in X$ and $k\in\mathbb N$ satisfy $x,\ldots,f^{k-1}(x)\in K_M\setminus K_{M+1}$ then $|(f^k)'x|\geq C^{-1}e^{\chi^+(\nu)k}$. Hence
  \begin{equation}\label{last-eq1}\frac{|V_{R,t}|}{|f^{n_q}(R)|}\leq 
  C\exp\left(-\chi^+(\nu)t\right)=C\exp\left(F(\nu)t\right).\end{equation}
Since $F(\delta_{O(f)})=0$ by the definition 
of the Lyapunov exponent \eqref{d-lyap},
we have $F(\nu)t\leq F(\nu_{R,t})t_q$.
By Proposition~\ref{koebe}, the distortion of $f^{n_q}$ on $R$ is bounded by the constant $D_q$ in \eqref{gamma-i} so that \eqref{last-eq1} yields \begin{equation}\label{last-eq2}\frac{|W_{R,t}|}{|R|}\leq CD_q\exp\left(F(\nu_{R,t}\right)t_q).\end{equation}

Let $\mu_{q}$ be a measure in the finite set
\[\{\nu_{R,t}\colon R\in\mathscr{R}_q(\mathbf A),\  0\leq t\leq t_q,\ W_{R,t}
\cap R_{q+1}(\mathbf A)\neq\emptyset\}\]
that maximizes the free energy within this finite set. Since $t_q>N_1$, from \eqref{last-eq2}  we obtain \[\begin{split} \frac{|R_{q+1}(\mathbf A)|}{|R_{q}(\mathbf A)|}&\leq\max_{R\in\mathscr{R}_q(\mathbf A)}\frac{|R\cap R_{q+1}(\mathbf A)|}{|R|}\leq\max_{R\in\mathscr{R}_q(\mathbf A)}\sum_{\substack{0\leq t\leq t_q\\ W_{R,t}\cap R_{q+1}(\mathbf A)\neq\emptyset}}\frac{|W_{R,t}|}{|R|}\\
&\leq  (t_q+1)\exp\left(F(\mu_q)t_q\right)\leq  \exp\left(\left( F(\mu_q)+\varepsilon\right)t_q \right) \end{split}\] provided $N_1$ is sufficiently large. 
This yields the inequality in Lemma~\ref{lem-step}(a) with $i=q$.

Since $t_q>N_1$,
similarly to the proof of Sublemma~\ref{esc-4} one can show that
 if $R\in\mathscr{R}_q(\mathbf A)$, $0\leq t\leq t_q$ and $W_{R,t}\cap R_{q+1}(\mathbf A)\neq\emptyset$ then $\int\vec{\phi}d\nu_{R,t}>\vec{\alpha}_q-\vec\varepsilon$ holds, 
provided $N_1$ is sufficiently large. This yields the desired inequality in Lemma~\ref{lem-step}(b) with $i=q$.
\medskip

\medskip

\noindent{\bf Case~3:} $m(f)=M$ and $\partial J_{m(f)}$ contains a one-sided attracting  periodic point 
(Remark~\ref{rem:C}(C-II)).
Set $\mu_q=\delta_{O(f)}$. The inequality in
 Lemma~\ref{lem-step}(a) with $i=q$ follows from $R_{q+1}(\mathbf A)\subset R_q(\mathbf A)$ and  $F(\mu_q)=0$. By Lemma~\ref{inclusion}(b-ii),
$
f^{n_q}(R_{q+1}(\mathbf A))$ is contained in $\cA_{t_q }(\vec{\phi},\vec{\alpha}_{q} -(1/2)\vec{\varepsilon})\cap K_M.$
Since $t_q >N_1$, 
if $N_1$ is sufficiently large then
the uniform convergence in \eqref{unic1} implies the inequality in 
Lemma~\ref{lem-step}(b) with $i=q$.
\medskip

\noindent{\bf Case~4:} $m(f)=M$ and $f$ has no attracting periodic point.
Recall that $K_{M }=\bigcup_{k=0}^{p_{M }-1}f^k(L_{M })$ by \eqref{KMF}.
Let $R\in\mathscr{R}_q(\mathbf A)$.  Lemma~\ref{inclusion}(b-iii) gives
 \[
f^{n_q}(R\cap R_{q+1}(\mathbf A))\subset \cA_{t_q }\left(\vec{\phi},\left(\vec{\alpha}_{q} -\frac{1}{2}\right)\vec{\varepsilon}\right)\cap K_{M}.\]
Since the unimodal map $f^{p_{M}}|_{L_{M }}$ 
is topologically exact,
the Lebesgue measure of the right set 
is estimated as follows by \cite[Proposition~4.4]{CRT}:
if $N_1$ is sufficiently large, then
there exists a measure
 $\nu_R\in \mathcal M(f)$ such that 
\begin{align}\label{last-eq101'}
 \frac{|\cA_{t_q }(\vec\phi,\vec{\alpha}_{q}
     -(1/2)\vec{\varepsilon})\cap K_{M}|}{|K_{M}|}
&\leq \exp\left(\left(F(\nu_R)+\frac{\varepsilon}{2}\right)t_q\right)\quad\text{and}\\
\label{hyp-300'}
\int\vec\phi d\nu_R&>
\vec{\alpha}_q- \vec{\varepsilon}.
\end{align}

For all sufficiently large $N_1$ we deduce that
\begin{equation}\label{obtain-eq}\begin{split}
\frac{|R\cap R_{q+1}(\mathbf A)|}{|R|}&\leq D_q\frac{|f^{n_q}
(R\cap R_{q+1}(\mathbf A))|}
     {|f^{n_q}(R)|}\\&\leq D_q\frac{|\cA_{t_q }(\vec\phi,
      \vec{\alpha}_{q}-(1/2)\vec{\varepsilon})\cap K_{M }|}{|f^{n_q}(R)|}\\
      &\leq \exp\left((F(\nu_{R})+\varepsilon)t_q\right).
\end{split}\end{equation}
The first inequality is because
 the distortion of $f^{n_q}$ on $R$ is bounded by the constant $D_q$ in \eqref{gamma-i} by Proposition~\ref{koebe}.
Since $t_q>N_1$,
the last inequality follows from
 \eqref{last-eq101'}
and the next uniform lower bound \[\frac{|f^{n_q}(R)|}{|K_{M }|}\geq \frac{\min_{0\leq k\leq p_{M }-1 }|f^k(L_{M })|}{|K_{M }|}>0.\]

Let $\mu_q$ be a measure in the finite set
$\{\nu_{R}\colon R\in\mathscr{R}_q(\mathbf A)\}$ that maximizes the free energy within this finite set. By \eqref{obtain-eq} we have
\[
\frac{|R_{q+1}(\mathbf A)|}{|R_q(\mathbf A)|}
\leq\max_{R\in \mathscr{R}_q(\mathbf A)}\frac{|R\cap R_{q+1}(\mathbf A)|}{|R|}\leq  \exp\left(\left(F(\mu_q)+\varepsilon\right)t_q\right),\] as required in Lemma~\ref{lem-step}(a) with $i=q$.
The inequality in
  Lemma~\ref{lem-step}(b) with $i=q$ is a consequence of \eqref{hyp-300'}. 
The proof of Lemma~\ref{lem-step} is complete.
\end{proof}

\noindent{\bf Step 3: Overall estimates.}
Let $N_1$ be sufficiently large so that the conclusions of Lemma~\ref{inclusion} and \ref{lem-step} hold
with $\varepsilon$ replaced by $\varepsilon/3$. Let $n>\max\{M,(q+1)N_1\}.$
 Put \[s(N_1)=\{0\leq i\leq q\colon t_i> N_1\}.\]
  By \eqref{ti-eq} we have $n=\sum_{i=0}^q t_i$, and so $s(N_1)\neq\emptyset$. 
For each $i\in s(N_1)$,
let $\mu_i$ be a measure in $\mathcal M(f)$ 
for which the conclusion of Lemma~\ref{lem-step} holds
with $\varepsilon$ replaced by $\varepsilon/3$.
Define a measure $\mu\in\mathcal M(f)$ by
\[
\mu=\left(\sum_{i\in s(N_1)}t_i\right)^{-1} \sum_{i\in s(N_1)}t_i\cdot\mu_i.
\]
Since
$R_{i+1}(\mathbf A)\subset R_{i}(\mathbf A)$ for $0\leq i\leq q$ as in \eqref{nest-seq} and
$(1/n)\sum_{i\in s(N_1)}t_i\to1$ as $n\to\infty$,
for all sufficiently large $n$ we have
\[\begin{split}
\frac{|R_{q+1}(\mathbf A)|}{|R_0(\mathbf A)|}&=
\prod_{i=0}^q\frac{|R_{i+1}(\mathbf A)|}{|R_i(\mathbf A)|}
\leq\prod_{i\in s(N_1)}
\frac{|R_{i+1}(\mathbf A)|}{|R_i(\mathbf A)|}\leq  \exp\left(\sum_{i\in s(N_1)}\left(F(\mu_i)+\frac{\varepsilon}{2}\right)t_i\right)\\
&\leq  \exp\left(F(\mu)(n-(q+1)N_1)+\frac{\varepsilon n}{2} \right)<  \exp\left((F(\mu)+\varepsilon)n\right).
\end{split}\]
The desired inequality in 
Proposition~\ref{over-lem}(a) follows from this estimate and the first inclusion in \eqref{nest-seq}.
By \eqref{v2} we have 
         \[\begin{split}
\sum_{i\in s(N_1)}t_i\int\vec\phi d\mu_i &\geq
\sum_{i\in s(N_1)}t_i\left(\vec\alpha_{i}-\frac{1}{3}\vec\varepsilon\right)
\\
&=\sum_{i=0}^q t_i\left(\vec\alpha_{i}-\frac{1}{3}\vec\varepsilon\right)-\sum_{i\notin s(N_1)}t_i\left(\vec\alpha_{i}-\frac{1}{3}\vec\varepsilon\right)
        \\
        &>n\left(\vec\alpha-\frac{2}{3}\vec\varepsilon\right)-(q+1)N_1\max_{0\leq i\leq q}\|\vec\alpha_{i}\|\vec1.
         \end{split}\]
 This yields the desired inequality in Proposition~\ref{over-lem}(b) for all sufficiently large $n$.
The proof of Proposition~\ref{over-lem} is complete.
\end{proof}

\subsection{Proof of Proposition~\ref{up}}\label{end}
Let $f\colon X\to X$ be an $S$-unimodal map with a non-flat critical point. 
We first treat the case $m(f)\geq1$.
Let $\varepsilon>0$. Let $N>M$ be an integer for which the conclusion of Proposition~\ref{over-lem} holds with $\varepsilon$ replaced by $\varepsilon/4$.
Let $n\geq N$ satisfy
 $\cA_n(\vec{\phi},\vec{\alpha})\neq\emptyset$. From \eqref{X-decomposition} and \eqref{claimeq2} 
we have
\[
\cA_n(\vec{\phi},\vec{\alpha})\subset \bigcup_{q=0}^{M}\bigcup_{\mathbf{t}\in
I_n(q) } 
 \bigcup_{\mathbf A\in M_{q+1,\ell}(\mathbf t,\mathbb Z_\varepsilon) }R(\mathbf{t},\mathbf A).
\]

Let $0\leq q\leq M$. By Proposition~\ref{over-lem}, for every
  $\mathbf{t}\in I_n(q)$ and every $\mathbf A\in M_{q+1,\ell}(\mathbf t,\mathbb Z_\varepsilon)$ with
$R(\mathbf{t},\mathbf A)\neq\emptyset$,
there exists a measure $\mu_{\mathbf{t},\mathbf A}\in\mathcal M(f)$ 
such that 
\begin{align}\label{zero3}
\left|R(\mathbf{t},\mathbf A)\right|&\leq 
\exp\left(\left(F(\mu_{\mathbf{t},\mathbf A})+\frac{\varepsilon}{4}\right) n\right)\quad\text{and}\\
\label{zero3.5}
\int\vec\phi d\mu_{\mathbf{t},\mathbf A}&>\vec\alpha-\frac{1}{4}\vec\varepsilon.\end{align}
If $n$ is sufficiently large, then by \eqref{v3} we have 
\[
\#M_{q+1,\ell}(\mathbf{t},\mathbb Z_\varepsilon)\leq\prod_{j=1}^\ell\left(\frac{3}{\varepsilon}\left(\sup\phi_j-\inf\phi_j\right)+2\right)^{q+1}\leq \exp\left(\frac{\varepsilon n}{4}\right).
\] 
Moreover we have
$\#I_n(0)= M+1$, and for $q\geq1$,
\[
\#I_n(q)\leq\begin{pmatrix}n\\q\end{pmatrix}
\begin{pmatrix}M+1\\q+1\end{pmatrix}
\leq n^q
\begin{pmatrix}M+1\\q+1\end{pmatrix}\leq \exp\left(\frac{\varepsilon n}{4}\right).
\]

Let $\mu\in\mathcal M(f)$ be a measure in the finite set \[\bigcup_{q=0}^M\left\{\mu_{\mathbf{t},\mathbf A}\colon
\mathbf{t}\in I_n(q),\ \mathbf A\in M_{q+1,\ell}(\mathbf t,\mathbb Z_\varepsilon), \ R(\mathbf{t},\mathbf A)\neq\emptyset\right\}\]
that maximizes the free energy within this finite set. 
Combining \eqref{zero3} with the upper bound on the number of all itineraries yields
\[\begin{split}
|\cA_n(\vec{\phi},\vec{\alpha})|&\le \sum_{q=0}^{M} \sum_{\mathbf{t}\in I_n(q)   }\sum_{\mathbf A\in M_{q+1,\ell}(\mathbf t,\mathbb Z_\varepsilon)  }
\left|R(\mathbf{t},\mathbf A)\right|\\
&\leq (M+1) \exp\left(\left(F(\mu)+\frac{\varepsilon}{4}\right)n\right)\max_{0\leq q\leq M}\#I_n(q)\#M_{q+1,\ell}(\mathbf{t},\mathbb Z_\varepsilon)\\
&\leq \exp\left((F(\mu)+\varepsilon)n\right),
\end{split}\]
as required in
Proposition~\ref{up}(a) provided $n$ is sufficiently large. The inequality in Proposition~\ref{up}(b)
is a consequence of \eqref{zero3.5}.

It is left to treat the case $m(f)=0$. If $\bar m(f)=1$, then $f$ has a two-sided attracting fixed point in ${\rm int}(X)$. The desired upper bound follows from the argument of Case~2
in the proof of Lemma~\ref{lem-step}. 
If $\bar m(f)=0$, then
$f|_{L_0}\colon L_0\to L_0$ is topologically exact.
The desired upper bound follows from the argument of Case~4 in the proof of Lemma~\ref{lem-step}. The proof of Proposition~\ref{up} is complete.
\qed

\section{Bimodal maps for which the level-2 LDP does not hold}
In this last section we prove Theorem~B. We start with a post-critically finite map with two  critical points, one periodic and the other not.
After a small perturbation we obtain a map for which the level-2 LDP does not hold.

\subsection{A one-parameter family of bimodal maps}\label{bimodal}
Let $X=[0,1]$ and $K=[1/2,1]$.
Let $f\colon X\to X$ be a $C^3$ map with negative Schwarzian derivative having only two 
critical points $c_0$ and $c_1$ with $0<c_0<1/2<c_1<1$, which are assumed to be non-degenerate: $f''(c_0)f''(c_1)\neq0$.
We assume the following conditions on $f$:
\begin{itemize}
    \item[(F1)] $f(X)\subset K$ and $c_1$ is a local maximal point of $f$;
    \item[(F2)] $c_1$ is periodic with prime period $3$.
\end{itemize}
\noindent
See \textsc{Figure}~\ref{fig2}.
Condition (F2) implies that the complement of the basin of the super-attracting periodic orbit of $c_1$ in $[f^2(c_1), f(c_1)]$ is a non-trivial hyperbolic basic set $\Lambda$.  Set $p=f(c_0)$. Fix a periodic point $q\in \Lambda$.
Further we assume 
\begin{itemize}
    \item[(F3)] 
    $p$ is a periodic point in $\Lambda$, and not contained in the periodic orbit of $q$.
\end{itemize}
\begin{remark} \label{rem:f}
{\rm From (F1), we have $f^2(X)\subset [1/2, f(c_1)]$ and so $
f^3(X)\subset  [f^2(c_1),f(c_1)]$.}
\end{remark}

We consider a $C^3$ parametrized family of maps $f_a\colon X\to X$, $a\in [-1,1]$ such that $f_0=f$ and the following conditions hold:
\begin{itemize}
\item[(F4)] $f_a|_K$ does not depend on $a\in [-1,1]$;
\item[(F5)] the critical point $c_0$ of $f$ remains to be a non-degenerate critical point of $f_a$ for $a\in [-1,1]$, and $\frac{d}{da} (f_a(c_0))|_{a=0}\neq 0$.
\end{itemize}
Let $\mu_{p}$ and $\mu_{q}$ denote the uniform probability distributions on the periodic orbits of $p$ and $q$ for $f$
respectively.
Let $\chi_p$ and $\chi_q$ denote their Lyapunov exponents. From  (F4), these do not depend on the parameter $a$.
For $x\in X$, $n\in\mathbb N$ and $a\in[-1,1]$, let $\delta_{a,x}^n$
denote the uniform probability distribution on the orbit $\{f_a^{j}(x)\}_{j=0}^{n-1}$.
Theorem~B follows from the next proposition. 
\begin{figure}
\begin{overpic}[scale=0.7]{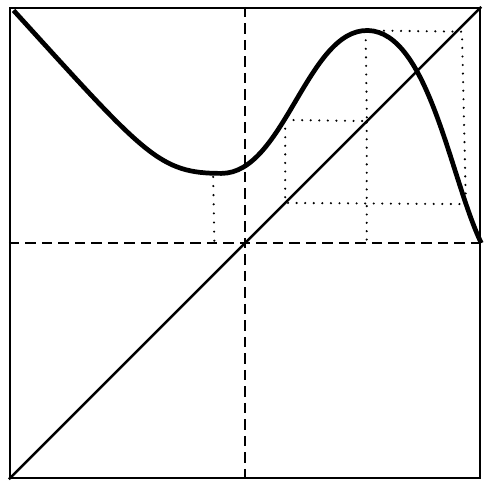}
\put(74,43){\footnotesize $c_1$}
\put(37,43){\footnotesize$c_0$}
\put(85,43){\footnotesize$
f(c_1)$}
\put(52,43){\footnotesize$
f^2(c_1)$}
\put(-2,-5){$0$}
\put(48,-5){$\frac12$}
\put(97,-5){$1$}
\put(-5,48){$\frac12$}
\put(-5,95){$1$}
\end{overpic}
\caption{The graph of the bimodal map $f_a$ in $\S$\ref{bimodal}}\label{fig2}
\end{figure}
\begin{prop}\label{abnormal}
There exists a parameter $a\in [-1,1]$ arbitrarily close to $0$ such that for any $\varepsilon>0$, 
there exist
open sets $\mathcal{G}_1$ and $\mathcal{G}_2$ in $\mathcal M$ such that:
\[
\mu_p\in \mathcal{G}_1\subset
\mathrm{cl}(\mathcal{G}_1) \subset \mathcal{G}_2;\]
\[\liminf_{n\to\infty}
\frac{1}{n}\log |\{x\in X\colon \delta_{a,x}^{n}\in \mathcal{G}_2\}|
\leq-\chi_p+\varepsilon;\]
\[\limsup_{n\to\infty}
\frac{1}{n}\log |\{x\in X\colon \delta_{a,x}^{n}\in {\rm cl}(\mathcal{G}_1)\}|
\geq-\frac{\chi_p}{2}.\]
\end{prop}

\begin{proof}[Proof of Theorem~B] Let $a\in[-1,1]$ be a parameter for which the conclusion of Proposition~\ref{abnormal} holds. 
If there were a rate function $I\colon\mathcal M\to[0,\infty]$ for $f_a$, 
we would have 
\[\begin{split}
{\limsup_{n\to\infty}}\frac{1}{n}\log|\{x\in X\colon\delta_{a,x}^n\in{\rm cl}(\mathcal G_1)\}|
&\leq -\inf_{\mathrm{cl}(\mathcal G_1)} I\le -\inf_{\mathcal G_2} I\\
&\leq
\liminf_{n\to\infty}\frac{1}{n}\log|\{x\in X\colon\delta_{a,x}^{n}\in\mathcal G_2\}|.
\end{split}\]
This contradicts the conclusion of Proposition~\ref{abnormal} for sufficiently small $\varepsilon>0$.
Therefore the level-2 LDP does not hold for 
$f_a$. \end{proof}

\subsection{Abnormal and normal large deviations rates} 
\label{perturbation}
For $\delta\in(0,1)$,
let $O_{\delta}(p)$ denote the open $\delta$-neighborhood of the periodic orbit of $p$ and put
\[
\mathcal{G}_\delta =\{ \mu\in \mathcal{M}\colon \mu(O_{\delta}(p)) >1-\delta\}.
\]
This is an open neighborhood of $\mu_p$  and we have $\mathrm{cl} (\mathcal{G}_{\delta'})\subset \mathcal{G}_\delta$ for $0<\delta'<\delta<1$.

The next two lemmas provide large deviations rates for neighborhoods of $\mu_p$.
For many $a$, 
we observe an abnormal rate as in Lemma~\ref{lm:abnormal}.
For many other $a$, 
we observe the normal rate as in Lemma~\ref{lm:normal}.
\begin{lemma}\label{lm:abnormal}
If $|a|$ is sufficiently small and   $f^m_a(c_0)=p$ for some $m\geq1$, then for any $\delta\in(0,1)$ we have
\[
\liminf_{n\to \infty}\frac{1}{n} \log |\{x\in X\colon \delta_{a,x}^n\in \mathcal{G}_\delta\}|\ge -\frac{\chi_{p}}{2}.
\]
\end{lemma}
\begin{proof}
Set $\gamma=\sqrt{|f_a''(c_0)|^{-1}\delta},$
and let $\varepsilon>0$. The condition $|x-c_0|\le \gamma e^{-n(\chi_p+\varepsilon)/2}$ 
 for $x\in X$ and sufficiently large $n\geq1$ implies $|f_a(x)-f_a(c_0)|\leq \delta e^{-n(\chi_{p}+\varepsilon)}$, and so $\delta_{a,x}^n\in \mathcal{G}_\delta$. Hence we have
\[
|\{x\in X\colon |x-c_0|<\gamma e^{-n(\chi_p+\varepsilon)/2}\}|=2\gamma e^{-(\chi_p+\varepsilon) n/2}\le 
|\{x\in X\colon\delta^n_{a,x}\in \mathcal G_\delta\}|.
\]
 This implies that
the lower limit in the lemma
is bounded from below by $-(\chi_p+\varepsilon)/2$. 
Since  $\varepsilon>0$ is arbitrary, the desired inequality follows.
\end{proof}

\begin{lemma}\label{lm:normal}
For any $\varepsilon>0$ there exists $\delta_*>0$ such that for any $\delta\in(0, \delta_*]$,
if $|a|$ is sufficiently small and  $f^m_a(c_0)=q$ for some $m\geq1$ then
\[\limsup_{n\to \infty}\frac{1}{n} \log |\{x\in X\colon \delta_{a,x}^n\in \mathcal{G}_{\delta}\}|\le -\chi_p+\varepsilon.
\]
\end{lemma}
\begin{proof}
Let $\varepsilon>0$. 
Since $f$ coincides with $f_a$ on $[f^2(c_1),f(c_1)]$, we have the large deviation estimate
\begin{equation}\label{eq:LDPonK}
\limsup_{n\to \infty}\frac{1}{n} \log |\{x\in [f^2(c_1),f(c_1)]\colon \delta_{a,x}^n\in \mathcal{G}_{\delta}\}|\leq -\chi_p+\frac{\varepsilon}{2},
\end{equation}
for $\delta>0$ small enough. This is a simplest case in the proof of Theorem~A,
and basically 
a known result.
Since  
 $f$ is locally diffeomorphic except at the two critical points, 
we see from Remark~\ref{rem:f} that  
for any open interval $U$ that contains $c_0$,
\begin{equation}\label{eq:XminusU}
\limsup_{n\to \infty}\frac1n \log |\{x\in X\setminus U\colon \delta_{a,x}^n\in \mathcal{G}_{\delta}\}|\leq -\chi_p+\frac{\varepsilon}{2}
\end{equation}
 provided $|a|$ is sufficiently small. Below we fix a small neighborhood $U$ of $c_0$, and 
consider orbits starting from $U$.

Let $\rho>0$. For an integer $k\geq0$,
let \[J_k=(q-\rho e^{-\chi_q k}, q+\rho e^{-\chi_q k}).\]
We assume $\rho$ is small enough so that $J_0\subset K$,
and the following hold for all $k\geq0$:
\begin{itemize}
\item the distortion  of $f^k$ on $J_k$ 
    is bounded by a constant independent of $k$;
    \item the Euclidean length of every connected component of $f^k(J_k\setminus J_{k+1})$ is bounded from below by a positive constant independent of $k$;
    \item $O_\delta(p)\cap\bigcup_{\ell=0}^kf^\ell(J_k)=\emptyset$.
    \end{itemize}
    Since $q$ is a hyperbolic repelling periodic point, all these conditions can be checked
by taking a $C^2$ linearization of $f$ in a neighborhood of the orbit of $q$.

Below we suppose that $|a|$ is sufficiently small and $m\ge 1$ is the smallest integer with $f^m_a(c_0)=q$. 
Recall that $f_a|_K=f|_K$.
The point $c_0$ is a non-degenerate critical point of $f^{m}_a$ because its orbit does not contain $c_1$. 
For $k\ge 0$, let $U_k$ be the connected component of $f^{-m}_a(J_k)$ that contains $c_0$. If $\rho>0$ is sufficiently small, there exists a constant $C_1>1$ depending on $m$ such that $C_1^{-1}e^{-\chi_q k/2}<|U_k|<C_1e^{-\chi_q k/2}$ for $k\ge 0$. 
The restriction of $f^m_a$ to each connected component of $U_k\setminus U_{k+1}$ for $k\geq0$ is a diffeomorphism onto a connected component of $J_k\setminus J_{k+1}$ with distortion bounded uniformly over all $k\ge 0$.

From the last property of $J_k$, if $x\in U_k$ and $m\le \ell \le m+k$ then $f^\ell_a(x)\notin O_{\delta}(p)$. Hence, if $\delta_{a,x}^n$  belongs to $\mathcal{G}_{\delta}$ for some $x\in U_{k}$ and $n\ge 0$, we have 
\begin{equation}\label{Gfinite}
m+ \#\{ \ell  \colon m+k\le \ell \le n-1, f^\ell_a(x) \in O_\delta(p)\} \ge (1-\delta) n. 
\end{equation}
Since the left-hand side is bounded by $n-k$, this happens only if $0\le k\le \delta n$.

The length of  the image $f^{m+k}_a(U_k\setminus U_{k+1})=f^k_a(J_k)$ is bounded from below by a constant independent of $k$. Hence, combining \eqref{eq:LDPonK} with the uniform distortion estimates for $f^{k}_a|_{J_k}$ and $f^m_a|_{U_k\setminus U_{k+1}}$, we conclude that
there exists a constant $C_2>1$ depending on $m$ such that
 the Lebesgue measure of the set of points in $U_k\setminus U_{k+1}$ for $0\le k\le \delta n$ satisfying  \eqref{Gfinite} 
is  bounded from above by $
C_2\exp\left((-\chi_p+\varepsilon/2)(1-\delta)n\right)$ provided $n$ is sufficiently large. 
Summing this bound
over all integers $k$ satisfying $0\leq k\leq\delta n$ yields
\[
\limsup_{n\to \infty}\frac1n \log |\{x\in U_0\colon \delta_{a,x}^n\in \mathcal{G}_{\delta}\}|\le  \left(-\chi_p+\frac{\varepsilon}{2}\right)(1-\delta)\leq
-\chi_p +\varepsilon.
\]
The last inequality holds if $\delta>0$ is
 sufficiently small. 
Since $\varepsilon>0$ is arbitrary, this together with \eqref{eq:XminusU} gives the conclusion. 
\end{proof}

\subsection{Proof of Proposition~\ref{abnormal}}\label{abnormal-pf}
Let $\varepsilon>0$. Let $\delta_*>0$ be as in Lemma~\ref{lm:normal} and let $\delta,\delta'\in(0,\delta_*)$ satisfy $\delta>\delta'$. 
Recall that the periodic points $p$, $q$ of $f_0$ are contained in its basic set $\Lambda$. From (F4) and (F5),
for $n\geq2$ we have
\begin{equation}\label{eq:transverse}
\frac{d}{da} (f_a^n(c_0))
=(f^{n-1})'(f_a(c_0)) \frac{d}{da} (f_a(c_0))\neq 0
\end{equation}
provided that $|a|$ is sufficiently small and 
$f^\ell(f_a(c_0))\neq c_1$
for every $0\leq\ell\leq n-2$. 
By the transitivity of $\Lambda$, if $|a|$ is sufficiently small and $f^m_a(c_0)=p$ (resp. $f^m_a(c_0)=q$) for some $m\ge 1$, we can find a parameter $a'$ arbitrarily close to $a$ such that $f^{m'}_{a'}(c_0)=q$ (resp. $f^{m'}_{a'}(c_0)=p$) for some $m'> m$.

We construct 
a sequence $\{a(i)\}_{i=0}^\infty$ of parameters 
in $[-1,1]$ converging to $0$, a sequence $\{\Delta(i)\}_{i=0}^\infty$ of positive reals  
converging to $0$, and strictly increasing sequences $\{m(i)\}_{i=0}^\infty$, $\{n(i)\}_{i=0}^\infty$ of positive integers
with the following properties:
\begin{itemize}
    \item[(a)] $f_{a(i)}^{m(i)}(c_0)=p$ for $i$ even, and
    $f_{a(i)}^{m(i)}(c_0)=q$ for $i$ odd;
    \item[(b)] if $a\in[-1,1]$ and $|a-a(i)|\le \Delta(i)$ then
    \[\begin{split}
    &\frac{1}{n(i)} \log |\{x\in X\colon \delta_{a,x}^{n(i)}\in \mathcal{G}_{\delta'}\}|> -\frac{\chi_p}{2}-\frac{1}{i}\quad \text{for $i$ even, }\\
    &\frac{1}{n(i)} \log |\{x\in X\colon \delta_{a,x}^{n(i)}\in \mathcal{G}_{\delta}\}|< -\chi_p+\varepsilon
    \ \quad \text{ for $i$ odd;}
\end{split}\]
        \item[(c)] for all integers $i$, $i'$ with $0\leq i'< i$,
    $|a(i')-a(i)|<\Delta(i')$. 
\end{itemize}
Then, Proposition~\ref{abnormal} holds for $f_a$ with  $a=\lim_{i\to \infty}a(i)$, 
$\mathcal{G}_1=\mathcal{G}_{\delta'}$ and 
$\mathcal{G}_2=\mathcal{G}_{\delta}$.

The construction is inductive. 
Start with $a(0)=0$, $m(0)=n(0)=1$ and 
a small number $\Delta(0)>0$ such that $\frac{d}{da} (f_a(c_0))\neq 0$ when $|a|<\Delta(0)$. Let $j\ge 1$ and suppose  
$a(i)$, $m(i)$, $n(i)$, $\Delta(i)$ have been constructed for $i=0,\ldots,j-1$.
As in the remark after \eqref{eq:transverse},  we 
take a parameter 
$a(j)$ close to $a(j-1)$ and an integer 
$m(j)>m(j-1)$ so that (a) and (c) hold for $i=j$. Further, we 
take a large integer $n(j)>n(j-1)$ so that one of the two conditions in 
(b) holds for $i=j$ and $a=a(j)$,  by Lemmas~\ref{lm:abnormal} or \ref{lm:normal} depending on the parity of $j$. Finally we take sufficiently small 
$\Delta(j)>0$ so that one of the two conditions in (b) holds for any $a\in[-1,1]$ with $|a-a(j)|<\Delta(j)$. 
Iterating this construction yields 
the four sequences 
with the required properties. 
\qed
\medskip

\begin{remark}\label{remark-final} 
The counterexample in Theorem~B can be found in  polynomial maps of degree $3$, by considering an appropriate one-parameter family $(f_a)$ of polynomials of degree $3$ and applying the argument parallel to the proof of 
Proposition~\ref{abnormal} above. Below we explain the construction of the one-parameter family $(f_a)$, leaving the details to interested readers. 

First we take a polynomial $f$ of degree $3$ so that there exist 
two non-degenerate critical points $c_0<c_1$ and non-empty compact intervals $K$ and $X$ with $K\subsetneq X$, $c_0\in X\setminus K$ and $c_1\in K$, and
a non-trivial hyperbolic basic set $\Lambda\subset K$, and a hyperbolic repelling periodic point $q\in \Lambda$ for which 
(F1), (F2), (F3) hold with $p=f(c_0)$. This is possible by the combinatorial argument on the dynamics of continuous piecewise monotone maps \cite[Chapter~II, $\S$4-6]{dMevSt93}.  

Second we consider a one-parameter family $(f_a)$ of polynomials of degree $3$ with $f_0=f$. 
Clearly (F4) is not true in general, but this is not essential as we put (F4) just to make computations easier. 
For $a\in\mathbb R$ sufficiently close to $0$,
let $c_0(a)<c_1(a)$ denote the critical points of $f_a$ and let $p(a)$
denote the continuation of the hyperoblic periodic point $p$ of $f_0$ respectively.
Instead of (F5) we require 
\begin{equation}
\label{eq:transversal}
\frac{d}{da} (f_a(c_0(a))-p(a))\big|_{a=0}\neq 0.
\end{equation}
In order to verify \eqref{eq:transversal},
we exploit the theory of Douady-Hubbard-Thurston on the dynamics of complex polynomials.
Consider the two-parameter family
\[
f_{\alpha,\beta}(x)=f(x)+\alpha x+\beta,
\]
and regard it as a family of complex dynamical systems with complex parameters $(\alpha,\beta)\in \mathbb{C}^2$. We apply \cite[Main~Theorem~1.1]{vStrien00} to this family, restricting $(\alpha,\beta)$ to a small neighborhood $V$ of $(0,0)$ in $\mathbb C^2$. To this end, we have to check that the family $(f_{\alpha,\beta})$ is a \textit{normalized family} (see \cite{vStrien00}), that is, no affine transformation conjugates two maps in the family $(f_{\alpha,\beta})$ with different parameters in $V$. This is not difficult: compute  $L^{-1}\circ f_{\alpha,\beta}\circ L$ for an affine map $L(z)=Az+B$ and use the fact that the coefficients of $f_{\alpha,\beta}$ with degree $2$ and $3$ do not depend on the parameter $(\alpha,\beta)$. As a conclusion, we obtain that 
the differential of the map
\[
\Psi\colon (\alpha,\beta)\in V \mapsto (f_{\alpha,\beta}(c_0(\alpha,\beta))- p(\alpha,\beta), f^3_{\alpha,\beta}(c_1(\alpha,\beta))-c_1(\alpha,\beta))\in \mathbb{C}^2
\]
at $(0,0)$ is injective, where $c_0(\alpha,\beta)$, $c_1(\alpha,\beta)$  denote 
the two non-degenerate critical points of $f_{\alpha,\beta}$ and
$p(\alpha,\beta)$ denotes the hyperbolic periodic point of $f_{\alpha,\beta}$ that are continuations of
$c_0$, $c_1$ and $p$ for $f_{0,0}$ respectively.
Clearly one can restrict $\Psi$ to $V\cap \mathbb{R}^2$ and see that $D\Psi_{(0,0)}\colon\mathbb{R}^2\to \mathbb{R}^2$ is injective. 
For all $a\in\mathbb R$ sufficiently close to $0$,  define $f_a=f_{\alpha(a), \beta(a)}$ by $\Psi(\alpha(a),\beta(a))=(a,0)$. Then 
\eqref{eq:transversal} holds for the family $(f_a)$.  

For the family $f_a$ constructed as above, we have   
\[
\frac{d}{da} (f^n_a(c_0(a))-f^{n-1}_a(p(a)))\big|_{a=0}= (f^{n-1}_0)'(f_0(c_0)) \cdot \frac{d}{da} (f_a(c_0(a))-p(a))\big|_{a=0}.
\]
Since $\frac{d}{da} f^{n-1}_a(p(a))\big|_{a=0}$ is uniformly bounded over all $n\geq1$,  
 \eqref{eq:transversal} implies that  $\frac{d}{da} f^n_a(c_0(a))\big|_{a=0}$ is comparable with $(f_0^{n-1})'(f_0(c_0))$ and grows exponentially in $n$. With this estimate and other conditions on $f=f_{0}$ discussed above, we can follow the argument in  
the proof of Proposition~\ref{abnormal} to obtain the counterexample in 
Theorem~B in the family $(f_a)$.  
\end{remark}

\subsection*{Acknowledgments}
We thank the referee for his or her very careful reading of the earlier versions of the manuscript and giving suggestions for improvements, which were indispensable for the completion of this paper. 
  We thank Juan Rivera-Letelier for  enlightening discussions throughout this project.
  We also thank Masayuki Asaoka, Miki U. Kobayashi for fruitful discussions.
  This paper was based on the authors' discussion  during their visits to the CIRM Luminy (Marseille) for participation 
in the conference `Dynamics Beyond Uniform Hyperbolicity' in May 2019.
HT was supported by the 
JSPS KAKENHI 25K21999, Grant-in-Aid for Challenging Research (Exploratory).
MT was supported by the JSPS KAKENHI 23K20806, Grant-in-Aid for Scientific Research (B).

\end{document}